\documentclass [11pt,a4paper]{article}
\usepackage[ansinew]{inputenc}
\usepackage{amssymb}
\usepackage{amsmath}
\usepackage{graphicx}
\usepackage{amsthm}
\usepackage{cite}
\usepackage{color}

\newtheorem{theorem}{\textbf{Theorem}}[section]
\newtheorem{lemma}{\textbf{Lemma}}[section]
\newtheorem{proposition}{\textbf{Proposition}}[section]
\newtheorem{corollary}{\textbf{Corollary}}[section]
\newtheorem{remark}{\textbf{Remark}}[section]
\newtheorem{definition}{\textbf{Definition}}[section]

\allowdisplaybreaks[4]

\def\be{\begin{equation}}
\def\ee{\end{equation}}
\def\bea{\begin{eqnarray}}
\def\eea{\end{eqnarray}}
\def\bt{\begin{theorem}}
\def\et{\end{theorem}}
\def\bl{\begin{lemma}}
\def\el{\end{lemma}}
\def\br{\begin{remark}}
\def\er{\end{remark}}
\def\bp{\begin{proposition}}
\def\ep{\end{proposition}}
\def\bc{\begin{corollary}}
\def\ec{\end{corollary}}
\def\bd{\begin{definition}}
\def\ed{\end{definition}}

\def\non{\nonumber }

\DeclareMathOperator{\tr}{tr}
\newcommand{\A}{\mathcal A}
\newcommand{\F}{\mathcal F}
\newcommand{\T}{\mathbb{T}}
\newcommand{\E}{\mathcal E}
\newcommand{\Id}{\mathbb I}

\newcommand\defeq{\stackrel{\scriptscriptstyle \text{def}}=}
\newcommand{\RR}{\mathbb R}

\begin{document}

\title{Global strong solutions of the full Navier-Stokes and $Q$-tensor system for nematic liquid crystal flows in $2D$: existence and long-time behavior}

\author{
  Cecilia Cavaterra
  \footnote{ Dipartimento di Matematica,
Universit\`a degli Studi di Milano, Via Saldini 50, 20133 Milano,
Italy. \texttt{cecilia.cavaterra@unimi.it}
 }
 \and
  Elisabetta Rocca
  \footnote{Elisabetta Rocca, Weierstrass Institute for Applied
Analysis and Stochastics, Mohrenstr.~39, 10117 Berlin,
Germany and
Dipartimento di Matematica, Universit\`a degli Studi di Milano, Via Saldini 50, 20133 Milano, Italy.
  \texttt{elisabetta.rocca@wias-berlin.de} and \texttt{elisabetta.rocca@unimi.it}
  }
  \and
  Hao Wu
  \footnote{School of Mathematical Sciences and Shanghai
Key Laboratory for Contemporary Applied Mathematics, Fudan
University, 200433, Shanghai, China.
    \texttt{haowufd@yahoo.com}}
\and
  Xiang Xu\footnote{
  Department of Mathematics and Statistics,
  Old Dominion University, Norfolk, VA, 23529, USA.
  \texttt{x2xu@odu.edu}}
}

\date{\today}
\maketitle

\begin{abstract}
We consider a full Navier-Stokes and $Q$-tensor system for incompressible liquid crystal flows of nematic type.
In the two dimensional periodic case, we prove the existence and uniqueness of global strong solutions that are
uniformly bounded in time. This result is obtained without any smallness assumption on the physical parameter
$\xi$ that measures the ratio between tumbling and aligning effects of a shear flow exerting over the liquid
crystal directors. Moreover, we show the uniqueness of asymptotic limit for each global strong solution as time
goes to infinity and provide an uniform estimate on the convergence rate.
\smallskip


\smallskip

\noindent
{\bf AMS (MOS) subject clas\-si\-fi\-ca\-tion:}~~35B44, 35D30, 35K45, 35Q30, 76A15.
\end{abstract}

\section{Introduction}
\setcounter{equation}{0}
In this paper, we study the global
well-posedness and long-time dynamics of a full coupled incompressible
Navier-Stokes and $Q$-tensor system due to Beris-Edwards \cite{BE94},
which models the evolution of incompressible liquid crystal flows of nematic type.
In the Landau-de Gennes theory \cite{dG93}, the local orientation and degree of ordering for the liquid crystal molecules
are characterized by a symmetric, traceless $d\times d$ tensor $Q$ (here $d$ stands for spatial dimension), which measures
the deviation of the second moment tensor from its isotropic value.
The $Q$-tensor can incorporate the biaxiality of the liquid crystal molecule alignment \cite{MZ10}.
Moreover, if $Q$ has two equal non-zero eigenvalues then it can be formally written as
$Q(x) = s(n(x)\otimes n(x)-{1\over d}\mathbb{I})$, with $s\in \mathbb{R}\setminus\{0\}$ and the vector $n: \mathbb{R}^d \to \mathbb{S}^{d-1}$ representing the averaged macroscopic molecular orientation,
so that the coupled $Q$-tensor system (see \eqref{navier-stokes}-\eqref{Q equ} below) reduces to the well-known Ericksen-Leslie system \cite{EL}.

In this paper, we restrict ourselves to the periodic case. Denote by $\mathbb{T}^{d}$ the periodic box with
period $a_i$ in the $i$-th direction and by $\mathcal{O}=(0,a_1)\times...\times(0,a_d)$ the periodic cell. Without loss of generality, we can simply set $\mathcal{O}=(0,1)^d$.
The coupled PDE system we are going to study consists of incompressible Navier-Stokes equations for the fluid velocity with highly nonlinear anisotropic
 force terms and nonlinear convection-diffusion equations of parabolic type that describe the evolution of the $Q$-tensor (see, e.g., \cite{PZ11}). More precisely, the full coupled Navier-Stokes and Q-tensor system takes the following form:
\begin{align}
u_t+u\cdot\nabla u-\nu\Delta
u+\nabla P &=\lambda \nabla\cdot(\tau+\sigma),\quad\, (x,t)\in \T^{d} \times \mathbb{R}^+,
\label{navier-stokes} \\
\nabla\cdot u &=0, \qquad\qquad \quad \ \ \ (x,t)\in \T^{d} \times \mathbb{R}^+,\label{incom}  \\
Q_t+u\cdot\nabla Q-S(\nabla{u}, Q)&=\Gamma{H}(Q),\qquad \quad\ (x,t)\in \T^{d} \times \mathbb{R}^+. \label{Q equ}
\end{align}
Here, the vector $u(x, t): \T^d \times (0, +\infty) \rightarrow \RR^d$ denotes the velocity field of the fluid and $Q(x, t): \mathbb{T}^d \times (0, +\infty) \rightarrow
S_0^{(d)}$  stands for the order parameter of liquid crystal molecules (see \eqref{S0} for the definition of the set $S_0^{(d)}$).  We assume that the system \eqref{navier-stokes}-\eqref{Q equ} is subject to the periodic boundary conditions
 \begin{equation}
 u(x+e_i,t) = u(x,t), \ \
Q(x+e_i,t)=Q(x,t), \ \ \mbox{for} \ (x,t) \in
\mathbb{T}^{d}\times\mathbb{R}^+,  \label{BC}
\end{equation}
where $\{e_i\}^d_{i=1}$ is the canonical orthonormal basis of $\mathbb{R}^d$. Moreover, the system is subject to initial spatially $1$-periodic data
 \begin{equation}
 u|_{t=0}=u_0(x)\ \ \mbox{with} \ \ \nabla\cdot u_0=0, \quad  Q|_{t=0}=Q_0(x),\quad \text{for}\ x\in \T^d.
 \label{IC}
 \end{equation}
We note that the system
preserves for all time both the symmetry and tracelessness of any solution $Q$ associated
to an initial datum with the same property \cite{PZ11,W12}.

The system \eqref{navier-stokes}-\eqref{Q equ} describes the complex interaction between the
fluid and the alignment of liquid crystal molecules: the evolution of the fluid affects the direction
and position of the molecules while changes in the alignment of molecules will also influence the
fluid velocity. The positive constants $\nu, \lambda$ and $\Gamma$ stand for the fluid
viscosity, the competition between kinetic energy and elastic potential
energy, and macroscopic elastic relaxation time (Deborah number) for
the molecular orientation field, respectively.

The free energy for the liquid crystal molecules is given by (see e.g., \cite{MZ10})
\begin{equation}\label{elastic energy}
\F(Q)\defeq\int_{\T^d}\left(\frac{L}{2}|\nabla{Q}|^2+f_B(Q)\right)\,dx.
\end{equation}
In the definition of $\mathcal{F}(Q)$, the gradient term corresponds to the elastic part of the free energy and $L>0$ is the elastic constant. Here, we simply use the so-called one constant approximation of the Oseen-Frank energy (cf. \cite{BM10}).  On the other hand, the bulk part $f_B(Q)$ of Landau-de Gennes type takes the following form
$$
f_B(Q)=\frac{a}{2}\tr(Q^2)-\frac{b}{3}\mathrm{tr}(Q^3)+\frac{c}{4}\tr^2(Q^2),
$$
where $a, b, c \in \RR$ are material-dependent and
temperature-dependent coefficients that are assumed to be constants
here. In particular, we assume that $$c>0,$$ which is necessary from
the modeling point of view to guarantee that the free energy $\F(Q)$
(and thus the total energy $\mathcal{E}(t)$ of the coupled system \eqref{navier-stokes}-\eqref{Q equ}) is bounded from
below (see \cite{M10, MZ10}).

The tensor $H=H(Q)$ in equation \eqref{Q equ} is related to the variational derivative of the free energy $\F(Q)$ with respect to $Q$ (under the constraint that $Q$ is both symmetric and traceless) such that
\begin{equation}
H(Q)\defeq -\frac{\partial \F(Q)}{\partial Q}=L\Delta{Q}-aQ+b\left(Q^2-\frac1d\mathrm{tr}(Q^2)\mathbb{I}\right)-cQ\tr(Q^2),\label{def of H}
\end{equation}
where $\mathbb{I} \in
\mathbb{R}^{d\times d}$ stands for the identity matrix. Then the matrix valued function $S(\nabla{u}, Q)$ in \eqref{Q equ} takes the following form
\begin{equation}
S(\nabla{u}, Q)\defeq (\xi D+\Omega)\big(Q+\frac1d \Id
\big)+\big(Q+\frac1d \Id \big)(\xi D-\Omega)-2\xi\big(Q+\frac1d \Id
\big)\tr(Q\nabla{u}),\label{S1}
\end{equation}
where $$D=\frac{\nabla u+\nabla^Tu}{2},\quad \Omega=\frac{\nabla
u-\nabla^Tu}{2}$$ represent the symmetric and skew-symmetric parts of
the rate of strain tensor, respectively. We note that $S(\nabla{u}, Q)$
describes the rotating and stretching effects on the order parameter
$Q$ due to the fluid, as the liquid crystal molecules can be tumbled
and aligned by the flow. In particular, the constant parameter $\xi\in\mathbb{R}$
in \eqref{S1} depends on the molecular shapes of the liquid crystal
and it is a measure of the ratio between the tumbling and the
aligning effect that a shear flow exerts on the liquid crystal
director.

Concerning the stress tensors $\tau$ and $\sigma$ on the right-hand side of Navier-Stokes equations \eqref{navier-stokes}, the symmetric part $\tau$
reads
\begin{equation}\label{symmetic tensor}
\tau\defeq-\xi\big(Q+\frac{1}{d}\Id\big)H(Q)-\xi{H}(Q)\big(Q+\frac{1}{d}\Id\big)+2\xi\big(Q+\frac{1}{d}\Id\big)\tr(QH(Q))
-L\nabla{Q}\odot\nabla{Q},
\end{equation}
in which the last term is understood as
$(\nabla{Q}\odot\nabla{Q})_{ij}=\sum_{k,l=1}^{d}\nabla_iQ_{kl}\nabla_jQ_{kl}$.
On the other hand, the skew-symmetric part $\sigma$ is given by
\begin{equation}\label{skew-symmetric tensor}
\sigma\defeq QH(Q)-H(Q)Q.
\end{equation}

We recall some related results in the literature.
The coupled Beris-Edwards system \eqref{navier-stokes}-\eqref{Q equ} has been recently studied by several authors.
For the simpler case $\xi=0$, which means that the molecules only tumble in a shear flow, but they are not
aligned by the flow (cf. \cite{PZ12}), the first contribution is due to \cite{PZ12}, in which  the authors proved the
existence of global weak solutions to the Cauchy problem in $\mathbb{R}^d$ with $d=2, 3$,
and they obtained higher global regularity as well as the
weak-strong uniqueness for $d=2$. Asymptotic behavior of the
Cauchy problem in $\mathbb{R}^3$ with $\xi=0$  is recently discussed in \cite{DFRSS14}.
Besides, initial boundary value problems subject to various boundary conditions for $d=2, 3$
have been investigated by several authors in \cite{ADL15, GR14, GR15} under the assumption $\xi=0$. In these works,
they proved the existence of global weak solutions, existence and
uniqueness of local strong solutions as well as some regularity criteria etc. For the full Navier-Stokes and $Q$-tensor system \eqref{navier-stokes}-\eqref{Q equ}
with general $\xi\in \mathbb{R}$,  existence of global weak solutions for the Cauchy
problem in $\mathbb{R}^d$ with $d=2, 3$ was established in
\cite{PZ11} for sufficiently small $|\xi|$, while the uniqueness of
weak solutions in the $2D$ setting is given quite
recently in \cite{DAZ}. On the other hand, in \cite{ADL14} the authors proved existence of global weak solutions
and local well-posedness with higher time-regularity for the initial boundary value problem
subject to inhomogeneous mixed Dirichlet/Neumann boundary conditions.

Some recent progresses have also been made on the mathematical analysis of certain modified versions
of the Beris-Edwards system in terms of its free energy. For instance, in \cite{W12}, the regular bulk
potential in \eqref{elastic energy} is replaced by a singular
potential of Ball-Majumdar type (cf. \cite{BM10}) that ensures the $Q$-tensor always
stays in the ``physical" region. Then, in the co-rotational regime $\xi=0$,
the author proved the existence of global weak solutions
for $d=2, 3$ and for $d=2$.  Moreover, he obtained the existence and uniqueness of global regular solutions.
In \cite{FRSZ14, FRSZ15}, non-isothermal variants of the Beris-Edwards system were derived and the authors proved
existence of global weak solutions in the case of a singular
potential under periodic boundary conditions for general $\xi$ and $d=3$.
In \cite{HD14}, the authors considered a general Beris-Edwards system where the Dirichlet type elastic functional as in
\eqref{elastic energy} is replaced by three quadratic functionals. For the Cauchy problem in $\RR^3$,
they proved existence of global weak solutions as well as the existence of a unique global strong solution provided that
the fluid viscosity is sufficiently large.
We also refer the interested readers to \cite{CX15, IXZ14} for well-posedness results regarding the $Q$-tensor gradient
flow generated by the general Landau-de Gennes energy with a cubic term (but without fluid coupling).

It is worth mentioning that a rigorous derivation from the
Beris-Edwards system (with general free energy and arbitrary $\xi$) to the classical Ericksen-Leslie system is recently given in
\cite{WZZ15} by using the Hilbert expansion method. We refer to \cite{CR, CRW, HLW14, LL01, WZZ14, WXL13} and the references therein for mathematical analysis
of the general Ericksen-Leslie system either under the unit length constraint of the molecule director or with Ginzburg-Landau approximation of the free energy.

In this paper, we are interested in the global well-posedness and long-time
behavior of the full Navier-Stokes and $Q$-tensor system \eqref{navier-stokes}-\eqref{IC}
in the two dimensional periodic setting. The main difficulty in handling
the current full coupled system with $\xi\in \mathbb{R}$ is due to the fact
that for $\xi\neq 0$ the system \eqref{navier-stokes}-\eqref{IC} no longer enjoys
the maximum principle for the $Q$-equation \eqref{Q equ}, which is
instead true in case $\xi=0$ (see e.g., \cite[Theorem 3]{GR15}). Due
to the loss of control on $Q$ in $L^\infty(0, T; L^\infty)$, extra
difficulties arise in obtaining estimates for highly nonlinear terms
of the system (see Proposition \ref{high2D}).
We note that a similar problem was encountered in \cite{PZ11} to prove the existence of global strong
solutions of the Cauchy problem in $\mathbb{R}^2$ (assuming that $|\xi|$ is sufficiently small). In order to get such highly
nonlinearities under control, the authors of \cite{PZ11} chose to work within the technical
Littlewood-Paley approach and then made use of the sharp
logarithmic Sobolev embedding of $H^{1+\epsilon}$ in $L^\infty$ (cf. \cite{BG}) together with the precise growth
of the constant of the Sobolev embedding of $H^1$ in $L^p$ for any $p>1$ (cf. \cite{CX}), and an optimal choice of
the non-constant index $p$ of interpolation depending on the norm of the solution. Then they established the existence of a unique global strong solution $(u, Q)$ to
the Cauchy problem in $\mathbb{R}^2$, whose $H^s\times H^{1+s}$-norm ($s>0$) may increases at most
quadruply exponential in time.

We note that in \cite{PZ11}, the smallness of the parameter $|\xi|$ is required because of the unboundedness of the whole plane $\mathbb{R}^2$,
which however can be removed in our current periodic setting (see \eqref{below}). In the periodic domain $\T^2$,
the first main result we are able to prove is the existence and uniqueness of global strong solutions $(u, Q)$ for arbitrary $\xi\in \mathbb{R}$,
whose $H^1\times H^2$-norm is indeed uniformly bounded in time (see Theorem \ref{strong2d}).
To achieve this goal, we use the idea of \cite{LL95} for the simplified liquid crystal system together with the interpolation techniques in \cite{PZ11} to derive a
suitable higher-order differential inequality for a specific quantity $\mathcal{A}(t)$ (see \eqref{def of A} for its definition),
which is essentially contained in the energy dissipation of the system \eqref{navier-stokes}-\eqref{IC} and is integrable with respect to time on the unbounded half line $\mathbb{R}^+$ such that $\mathcal{A}(t)\in L^1(0,+\infty)$ (see Proposition \ref{high2D}). The resulting higher-order energy inequality \eqref{hiA} has a delicate double-logarithmic type structure and it plays a crucial role in three aspects of the subsequent proofs:
(1) it yields uniform-in-time estimates on $H^1\times H^2$-norm of the global strong solution $(u,Q)$ provided that $(u_0,Q_0)\in H^1\times H^2$ (see \eqref{highN1}); (2) it implies the decay of $\mathcal{A}(t)$ to zero as $t\to+\infty$ and thus gives a characterization of the $\omega$-limit set of the evolution system \eqref{navier-stokes}-\eqref{IC} (see Lemma \ref{lemma on decay property}); (3) it helps to obtain an uniform estimate on the rate of convergence to equilibrium for the global strong solution (see \eqref{ddecaA}).

Our second main result is about the long-time behavior of global
strong solutions obtained in Theorem \ref{strong2d} (see Theorem \ref{long-time}).
 The problem whether a bounded global solution of a nonlinear
evolution equation will converge to a single equilibrium as
time tends to infinity is of great importance. This issue is nontrivial since the
structure of the equilibrium set may form a continuum for many
dynamic systems in higher space dimensions. For instance, under current
periodic boundary conditions, it is
expected that the dimension of the equilibrium set of our problem
\eqref{navier-stokes}-\eqref{IC} is at least $2$  due to the simple fact that a shift in
each variable produces another steady state. Hence, it is interesting
to determine whether a trajectory defined by a global strong solution
will converge to a single
steady state or not. To this end, we first construct a specific gradient inequality
for tensor valued functions subject to
periodic boundary conditions (see Lemma \ref{lemma on LS inequality}), then we apply the \L
ojasiewicz-Simon approach (see \cite{S83}  and also \cite{FS})
to achieve our aim. This approach turns out to be a powerful tool in
the study of long-time dynamics of evolution equations, and we refer interested
readers to \cite{H06} and the references therein for various
applications.

The rest of this paper is organized as follows. In Section 2 we
introduce the notations as well as some preliminary results,
and then state the main results of this paper.
Section 3 is devoted to the derivation of a specific higher-order differential inequality
that will be crucial in the subsequent proof.
In Section 4, we prove the existence and uniqueness of global strong solutions to the
Beris-Edward system \eqref{navier-stokes}-\eqref{IC}.
In Section 5 we demonstrate that every global strong solution will
converge to a single equilibrium and provide a uniform estimate on the convergence rate.
Some detailed calculations will be presented in the Appendix Section.

\section{Preliminaries and Main Results}
\setcounter{equation}{0}

\subsection{Notations}
Let $X$ be a real Banach space with norm $\|\cdot\|_X$ and $X^*$ be its dual space. By $<\cdot,\cdot>_{X^*,X}$ we indicate the duality product
between $X$ and $X^*$.
We denote by $L^p(\T^d, M)$, $W^{m,p}(\T^d, M)$ the usual Lebesgue and Sobolev spaces defined on the torus $\T^d$ for $M$-valued functions
(e.g., $M=\mathbb{R}$, $M=\mathbb{R}^d$ or $M=\mathbb{R}^d\times \mathbb{R}^d$) that are in $L^p_{loc}(\mathbb{R}^d)$ or $W^{m,p}_{loc}(\mathbb{R}^d)$
and periodic in $\T^d$, with norms denoted by $\|\cdot\|_{L^p}$, $\|\cdot\|_{W^{m,p}}$, respectively.
For $p=2$, we simply denote $H^m(\T^d)= W^{m,2}(\T^d)$ with norm $\|\cdot \|_{H^m}$. In particular for $m=0$, we denote $H^0(\T^d)=L^2(\T^d)$ and the inner product on $L^2(\T^d)$ will be denoted by $(\cdot, \cdot)_{L^2}$.
For simplicity, we shall not distinguish functional spaces when scalar-valued, vector-valued or matrix-valued functions are involved
if they are clear from the context.

Einstein summation convention will be used throughout this paper.
For arbitrary vectors $u, v\in\RR^d$, we denote $u\cdot v=u_i v_i$ the inner product in $\RR^d$.
For any matrix $Q\in \RR^{d\times d}$, we use the Frobenius norm $|Q| =\sqrt{\tr(Q^2)}=\sqrt{Q_{ij}Q_{ij}}$.
Let $S_0^{(d)}$ denote the space of symmetric traceless matrices with spatial dimension $d$,
\begin{equation}
  S_0^{(d)}\defeq\Big\{Q \in \RR^{d \times d}\ | \ Q_{ij}=Q_{ji}, \, \tr(Q)=0, \, i,j=1,...,d  \Big\}.\label{S0}
\end{equation}
Then for matrices $A, B \in S_0^{(d)}$, we denote $A: B=\tr(AB)$.
Concerning the norms for derivatives, we
denote $|\nabla Q|^2(x)=\nabla_kQ_{ij}(x)\nabla_kQ_{ij}(x)$ and
$|\Delta Q|^2(x)=\Delta{Q}_{ij}(x)\Delta{Q}_{ij}(x)$. Sobolev spaces
for $Q$-tensors will be defined in terms of the above norms. For instance,
$L^2(\T^d, S_0^{(d)}) = \{Q: \T^d\rightarrow S_0^{(d)},
\int_{\T^d}|Q(x)|^2\,dx < \infty \}$ and  $H^1(\T^d, S_0^{(d)}) =
\{Q: \T^d\rightarrow S_0^{(d)}, \int_{\T^d}|\nabla
Q(x)|^2+|Q(x)|^2\,dx < \infty \}$ etc. Concerning the divergence of a
$d\times d$ differentiable matrix-valued function $\sigma=(\sigma_{ij})$, its
$i$-th component is given by
$(\nabla\cdot\sigma)_i=\nabla_j\sigma_{ij}$, $1\leq i, j\leq d$.

For any normed space $X$, the subspace of functions in $X$ with zero-mean will be denoted by $\dot X$,
that is $\dot X=\left\{v\in X: \int_{\T^d} v \,dx=0\right\}$. Then we recall the well established functional settings
for periodic solutions to Navier-Stokes equations (see e.g., \cite{Te}):
\begin{align*}
\mathbf{H}&=\{v \in \dot{L}^2(\T^d, \RR^d) ,\ \nabla\cdot v=0\},\\
\mathbf{V}&=\Big\{v\in \dot{H}^1(\T^d,\RR^d), \; \nabla\cdot v=0\Big\},\\
\mathbf{V}'&=\text{the\ dual space of\ } \mathbf{V}.
\end{align*}
In the spatial periodic setting, one can define a mapping $\mathcal{S}$ associated
with the Stokes problem:
 \begin{equation}
 \mathcal{S} u=-\Delta u, \quad  \forall\, u\in D(\mathcal{S})\defeq\{u\in \mathbf{H}, \, \Delta u\in \mathbf{H}\}=\dot H^2(\T^d, \RR^d)\cap
 \mathbf{H}.\label{stokes}
\end{equation}
 The operator $\mathcal{S}$ can be seen as an unbounded
positive linear self-adjoint operator on $\mathbf{H}$. If $D(\mathcal{S})$ is endowed
with the norm induced by $\dot{H}^2(\T^d)$, then $\mathcal{S}$ becomes an
isomorphism from $D(\mathcal{S})$ onto $\mathbf{H}$. For detailed properties of $\mathcal{S}$, we refer to \cite{Te}.

We denote by $C$ a generic constant that may depend on $\nu$,
$\Gamma$, $\lambda$, $\xi$, $L$, $a$, $b$, $c$, $\T^d$ and the
initial data $(u_0, Q_0)$, whose value is allowed to vary on
occurrence. Specific dependence will be pointed out explicitly
if necessary.

\subsection{Basic energy law and global weak solutions}
We first present some basic properties of the Navier-Stokes and $Q$-tensor system
\eqref{navier-stokes}-\eqref{IC} that are valid in both two and three dimensional cases.

The total energy of the system \eqref{navier-stokes}-\eqref{IC} consists of two parts: the kinetic
energy for the velocity field $u$ and the free energy $\F(Q)$ (see \eqref{elastic energy}). More
precisely, we have
\begin{equation}\label{total energy}
\E(t)\defeq \frac12\int_{\T^d} |u|^2(x,t)\,dx+\lambda \F(Q(t)).
\end{equation}
By the same argument as in \cite[Proposition 1]{PZ11} for the whole space case in $\RR^d$ (see also \cite[Proposition 2.1]{DAZ}),
we can derive the following basic energy law:

\begin{lemma}[Basic energy law] \label{prop on basic energy law}
Suppose $d=2, 3$ and $(u, Q)$ be a smooth solution to the problem \eqref{navier-stokes}-\eqref{IC}. Then we have
\begin{equation}\label{basic energy law}
\frac{d}{dt}\E(t)=-\nu\int_{\T^d} |\nabla{u}|^2\,dx-\lambda \Gamma\int_{\T^d} |H(Q)|^2\,dx\leq 0,\quad \forall\, t>0.
\end{equation}
\end{lemma}

Lemma \ref{prop on basic energy law} reflects the energy dissipation
of the liquid crystal flow and indicates that the energy functional $\mathcal{E}(t)$ which is bounded from below since $c>0$,
serves as a Lyapunov functional for the system \eqref{navier-stokes}-\eqref{IC}. This property provides necessary uniform estimates
for further mathematical analysis of the PDE system \eqref{navier-stokes}-\eqref{IC}, for instance, the existence of global weak solutions.

\begin{lemma}\label{prop on lower order norm}
Suppose $d=2, 3$. Let $(u, Q)$ be a smooth solution to the problem \eqref{navier-stokes}-\eqref{IC}. Then we have
\begin{equation}\label{lower order bound}
\|u(t)\|_{L^2}+\|Q(t)\|_{H^1}\leq C, \qquad \forall\, t>0,
\end{equation}
where the constant $C>0$ depends on $\|u_0\|, \|Q_0\|_{H^1}, L, \lambda, a, b, c$ and
$\T^d$. Moreover, it holds
\begin{equation}
\int_0^T\int_{\T^d} |\nabla{u}(x,t)|^2+ |\Delta{Q}(x, t)|^2\,dxdt < C_T,\qquad\forall\, T>0,\label{intes1}
\end{equation}
where $C_T>0$ may further depend on $\nu$, $\Gamma$ and $T$.
\end{lemma}
\begin{proof}
It follows from Lemma \ref{prop on basic energy law}
that
\begin{equation}
  \mathcal{E}(t) +\int_0^t \int_{\T^d} \nu |\nabla u|^2 + \lambda\Gamma |H(Q)|^2 dxdt= \mathcal{E}(0), \quad  \forall\, t>0.\label{es1}
\end{equation}
We easily infer from the Sobolev embedding theorem ($d=2,3$) that
$$\mathcal{E}(0)=\frac12\|u_0\|_{L^2}^2+\lambda \F(Q_0)\leq C(\|u_0\|_{L^2}, \|Q_0\|_{H^1}).$$
On the other hand, there exists a constant $M=M(a,b,c)>0$ large enough (see \cite[(18)]{PZ11}) such that
$$\frac{M}{2}\mathrm{tr}(Q^2)+\frac{c}{8}\mathrm{tr}^2(Q^2)\leq \left(M+\frac{a}{2}\right)\mathrm{tr}(Q^2)-\frac{b}{3}\mathrm{tr}(Q^3)
+\frac{c}{4}\mathrm{tr}^2(Q^2),$$
which combined with the Young's inequality and the fact $c>0$  yields that
\begin{eqnarray}
\frac{a}{2}\mathrm{tr}(Q^2)-\frac{b}{3}\mathrm{tr}(Q^3)+\frac{c}{4}\mathrm{tr}^2(Q^2)&\geq & -\frac{M}{2}\mathrm{tr}(Q^2)+\frac{c}{8}\mathrm{tr}^2(Q^2)\nonumber\\
&\geq& \frac{1}{2}\mathrm{tr}(Q^2)+ \frac{c}{16}\mathrm{tr}^2(Q^2)-\frac{(M+1)^2}{c}.\nonumber
\end{eqnarray}
Then we have following estimate
\begin{align}
& \frac12\|u(t)\|_{L^2}^2+\frac{\lambda L}{2}\|\nabla{Q}(t)\|_{L^2}^2+\lambda \int_{\T^d} \frac{1}{2}\mathrm{tr}(Q^2(t))+\frac{c}{16}\tr^2(Q^2(t))\,dx\nonumber\\
\leq&\ \frac12\|u(t)\|^2+\lambda \F(Q(t))+\frac{\lambda (M+1)^2}{c}|\T^d|,\label{below}
\end{align}
where $|\T^d|$ stands for the Lebesgue measure of $\T^d$. As a consequence, we can deduce that $\mathcal{E}(t)$ is uniformly bounded from below
by a generic constant only depending on the coefficients $a, b, c, \lambda$ and the size of periodic domain.
Hence, the estimate \eqref{lower order bound} easily follows from \eqref{es1} and \eqref{below}.

Next, we infer from \eqref{lower order bound}, \eqref{es1}, \eqref{below} and the Sobolev embedding theorem $(d=2,3$) that
\begin{align*}
&\int_0^t\int_{\T^d}\big|L\Delta{Q}\big|^2\,dxd\tau \\
\leq&\ 2\int_0^t\int_{\T^d}\big|H(Q)\big|^2\,dxd\tau
+2\int_0^t\int_{\T^d}\left|-aQ+b\left(Q^2-\frac13\mathrm{tr}(Q^2)\mathbb{I}\right)-cQ\tr(Q^2)\right|^2\,dxd\tau \\
\leq&\  2\int_0^t\int_{\T^d}\big|H(Q)\big|^2\,dxd\tau+C \int_0^t \left(\|Q(\tau)\|_{L^2}^2+ \|Q(\tau)\|_{L^4}^4+\|Q(\tau)\|_{L^6}^6\right)d\tau \\
\leq&\  \frac{2}{\lambda\Gamma}\big(\mathcal{E}(0)-\mathcal{E}(t)\big)+Ct\\
\leq&\  C(1+t), \quad \forall\, t>0,
\end{align*}
where $C$ depends on $\|u_0\|, \|Q_0\|_{H^1}, \Gamma, L, \lambda, a, b, c$ and
$\T^d$. Then the conclusion \eqref{intes1}  follows from the above estimate and \eqref{es1}.
\end{proof}
\begin{remark}
For the full Navier-Stokes and Q-tensor system \eqref{navier-stokes}-\eqref{IC} with general $\xi\in \mathbb{R}$,
existence of weak solutions for the Cauchy problem in the whole space $\mathbb{R}^d$ with $d=2 ,3$ is
established in \cite{PZ11} for sufficiently small $\xi$. On the other hand, for the initial boundary value problem
in a bounded domain in $\RR^d$, in \cite{ADL14} existence of global weak solutions under inhomogeneous mixed Dirichlet/Neumann boundary conditions
were obtained without any restriction on $\xi$.
The smallness for $\xi$ can be removed for the bounded domain case because one can use a generic constant depending on the domain size
to get a priori $L^2$ estimates for the $Q$-tensor (see \eqref{below}).
\end{remark}

Since we are working with the periodic domain, the following result can be easily  proved in a way similar to \cite{ADL14}:

\begin{proposition}[Existence of global weak solutions] \label{weakexe}
Suppose that $d=2,3$ and $\xi\in \RR$. For any initial data $(u_0, Q_0)\in \mathbf{H}\times H^1(\T^d, S_0^{(d)})$, the problem \eqref{navier-stokes}-\eqref{IC} possesses at least one global-in-time weak solution $(u,Q)$ such that
\begin{align}
&u \in L^\infty(0, T; \mathbf{H}) \cap L^2(0, T; \mathbf{V}),  \\
&Q \in L^\infty(0, T; H^1(\T^d,S_0^{(d)})) \cap L^2(0, T; H^2(\T^d,S_0^{(d)})).
\end{align}
Moreover, for a.e. $t\in (0, T)$, the following energy inequality holds:
\begin{equation}
  \mathcal{E}(t) +\int_0^t \int_{\T^d} \nu |\nabla u|^2 + \lambda\Gamma |H(Q)|^2 dxdt\leq \mathcal{E}(0).\label{BELin}
\end{equation}
\end{proposition}

\subsection{Main results}

In the remaining part of this paper, we shall focus on the two dimensional case that $d=2$. First, we observe the simple fact that when $d=2$,
it holds $\tr(Q^3)=0$ and thus the cubic term with coefficient $b$ in the free energy $\F(Q)$ (see \eqref{elastic energy})
 vanishes (cf. \cite{IXZ14}). As a consequence, we have a simpler expression for $H(Q)$ in the $2D$ case:
\begin{equation}
H(Q)=L\Delta{Q}-aQ-cQ\tr(Q^2).\label{def of H2d}
\end{equation}

Let us introduce the notion of strong solutions to problem \eqref{navier-stokes}-\eqref{IC}:

\begin{definition}
\label{def of strong}
Suppose that $d=2$ and $(u_0, Q_0)\in \mathbf{V}\times H^2\big(\T^2, S_0^{(2)}\big)$. A pair $(u, Q)$ is called a global
strong solution to problem \eqref{navier-stokes}-\eqref{IC} if
\begin{align}
&u \in C([0, +\infty); \mathbf{V}) \cap L_{loc}^2(0, +\infty; H^2(\T^2,\RR^2)), \\
&Q \in C([0, +\infty); H^2(\T^2,S_0^{(2)}) \cap L_{loc}^2(0, +\infty;
H^3(\T^2,S_0^{(2)})).
\end{align}
Moreover, the equations \eqref{navier-stokes} for $u$ and the equations \eqref{Q equ} for $Q$ are satisfied in $L_{loc}^2(0, +\infty; \mathbf{H})$ and $L^2_{loc}(0, +\infty; L^2(\T^2,S_0^{(2)}))$, respectively.
\end{definition}

Then we state the main results of this paper. The first result
is about the global well-posedness of the hydrodynamic system \eqref{navier-stokes}-\eqref{IC} in $\T^2$.

\begin{theorem}[Existence and uniqueness of global strong solutions] \label{strong2d}
Suppose $d=2$ and $\xi\in\RR$. Then, for any $(u_0, Q_0)\in
\mathbf{V}\times H^2\big(\T^2, S_0^{(2)}\big)$, problem
\eqref{navier-stokes}-\eqref{IC} admits a unique global solution
$(u, Q)$ in the sense of Definition \ref{def of strong}, which satisfies
$$\|u(t)\|_{H^1}+\|Q(t)\|_{H^2}\leq C, \quad \forall\, t\geq 0,$$
where  $C>0$ is a constant that depends on $\nu, \Gamma, L, \lambda, a, c,
\T^2, \|u_0\|_{H^1}, \|Q_0\|_{H^2}$ and $\xi$.
\end{theorem}

Our second main result states that for any global strong solution obtained in Theorem \ref{strong2d}, it has an unique asymptotic limit as $t\to +\infty$.

\begin{theorem}[Uniqueness of asymptotic limit] \label{long-time}
Suppose that the assumptions in Theorem \ref{strong2d}
are satisfied. For any $(u_0, Q_0)\in
\mathbf{V}\times H^2\big(\T^2, S_0^{(2)}\big)$, the unique global strong solution to problem
\eqref{navier-stokes}-\eqref{IC} converges
to a single steady state solution $(0, Q_\infty)$ as time tends to
infinity:
\begin{equation}
\displaystyle\lim_{t\rightarrow+\infty}(\|u(t)\|_{H^1}+\|Q(t)-Q_\infty\|_{H^2})=0,
\end{equation}
where $Q_\infty\in S_0^{(2)}$ satisfies the elliptic problem in $\T^2$
$$
  L\Delta{Q}_\infty -aQ_\infty-c\tr(Q_\infty^2)Q_\infty=0,
  \;\;\mbox{in } \T^2, \quad Q_\infty(x+e_i)=Q_\infty \;\;\mbox{for } x\in \T^2.
$$
Furthermore, the following estimate on convergence rate holds
\begin{equation}
\|u(t)\|_{H^1}+\|Q(t)-Q_\infty\|_{H^2}\leq
C(1+t)^{-\frac{\theta}{1-2\theta}}, \quad\forall\, t\geq 0.\label{covrate}
\end{equation}
Here, $C>0$ is a constant that depends on $\nu, \Gamma, L, \lambda, a, c, \xi,
\T^2, \|u_0\|_{H^1}, \|Q_0\|_{H^2}, \|Q_\infty\|_{H^2}$, and
$\theta\in (0, \frac12)$  is the constant given in Lemma \ref{lemma on LS inequality} depending on $Q_\infty$.
\end{theorem}

\section{Higher-Order Energy Inequality}
\setcounter{equation}{0}

In this section we will derive a useful higher-order energy
inequality for problem \eqref{navier-stokes}-\eqref{IC}. For
the sake of simplicity, the subsequent calculations shall be performed
formally on smooth solutions of the problem
\eqref{navier-stokes}-\eqref{IC}, without referring to any
approximation. Nevertheless, they can be justified by working within
the Faedo-Galerkin approximation scheme \eqref{approx
navier-stokes}-\eqref{approx IC} given in Section 4.

We start by recalling the following special cases of the Gagliardo-Nirenberg inequality in 2D that will be frequently used
in the subsequent proofs (see, e.g., \cite{LA1}):
\begin{lemma}
Suppose $d=2$.  We have
\begin{eqnarray}
\|g\|_{L^4}&\leq& C\big(\|\nabla g\|_{L^2}^\frac12\|g\|_{L^2}^\frac12 +\|g\|_{L^2}\big), \quad \forall\, g\in H^1(\T^2),\label{GN}\\
\|\nabla g\|_{L^2} &\leq& \|g\|_{L^2}^\frac12\|\Delta g\|_{L^2}^\frac12, \quad \forall\, g\in H^2(\T^2).\label{GN1}
\end{eqnarray}
\end{lemma}
Besides, we will make use of the following $L^p$-interpolation inequality with precise growth of the
constant in 2D, which follows from \cite{CX} (see also \cite[Lemma 10]{MP}) and the Sobolev extension theorem \cite[Chap. 2, Sect. 3.6]{Ne}:
\begin{lemma}\label{Lp}
Suppose $d=2$. For any $\eta>1$, it holds:
$$\|g\|_{L^{2\eta}}\leq C\sqrt{\eta}\|g\|_{H^1}^{1-\frac{1}{\eta}}\|g\|_{L^2}^\frac{1}{\eta}, \quad \forall\, g\in H^1(\T^2), $$
where the constant $C$ is independent of the exponent $\eta$ and function $g$.
\end{lemma}
Next, we recall that when $\xi=0$ the system
\eqref{navier-stokes}-\eqref{IC} enjoys a maximum principle for the
$Q$-equation \eqref{Q equ} (see e.g., \cite[Theorem 3]{GR15}).
However, since now the parameter $\xi$ is allowed to be non-zero,
the maximum principle property is no longer valid. The loss of
control on $Q\in L^\infty(0, T; L^\infty)$ brings us several
difficulties in obtaining estimates for highly nonlinear terms of
the system. In order to handle the $L^\infty$-norm of $Q$, we shall
use the following well-known results

\begin{lemma}[Agmon's Inequality \cite{Te97}]
When $d=2$, it holds
\begin{equation}
\|g\|_{L^\infty}\leq C\|g\|_{L^2}^\frac12\|g\|_{H^2}^\frac12, \quad \forall\, g\in H^2(\T^2).\label{AG}
\end{equation}
\end{lemma}
\begin{lemma}[Br\'{e}zis-Gallouet Inequality \cite{BG}]
When $d=2$, for any $g\in H^2(\T^2)$, it holds
\begin{equation}
 \|g\|_{L^\infty}\leq C\|g\|_{H^1}\sqrt{\ln (1+\|g\|_{H^2})} +C\|g\|_{H^1}.\label{BG}
 \end{equation}
\end{lemma}

Now we state the main result of this section.
\begin{proposition}\label{high2D}
Let $d=2$ and
\begin{equation} \label{def of A}
\A(t)=\|\nabla{u}(t)\|_{L^2}^2+\lambda\|H(Q(t))\|_{L^2}^2.
\end{equation}
For any $\xi\in \RR$, the following energy inequality holds:
\begin{eqnarray}
&&\frac{d}{dt}\A(t) + \frac{\nu}{2}\|\Delta{u}(t)\|_{L^2}^2+ \frac{\lambda\Gamma}{2}\|\nabla{H}(Q(t))\|_{L^2}^2\non\\
&\leq& C_\ast\Big[1+|\xi|\big[1+\ln(e+\ln(e+\A(t)))\big](e+\ln(e+\A(t)))\Big]\big[e+\A(t)\big]\A(t),\label{hiA}
\end{eqnarray}
where $C_\ast>0$ is a constant that depends on $\nu, \Gamma, L, \lambda, a, c,
\T^2, \|u_0\|_{L^2}, \|Q_0\|_{H^1}$ and $\xi$.
\end{proposition}
%
%
\begin{proof}
After a lengthy calculation (see Appendix for details), we obtain
\begin{eqnarray}
&&\frac12\frac{d}{dt}\A(t)+ \nu\|\Delta{u}\|_{L^2}^2+\lambda \Gamma\|\nabla H\|_{L^2}^2\non\\
& =& \int_{\T^2}\big(u\cdot\nabla{u},\Delta{u}\big)\,dx
-2\lambda \int_{\T^2} \nabla_l u_k\nabla_l\nabla_kQ_{ij} H_{ij} dx
+\frac{\lambda}{L}\int_{\T^2} u_k\nabla_k F_{ij} H_{ij} dx\non\\
&& - 2\lambda \int_{\T^2} \nabla_j u_i(\nabla_lQ_{kj}\nabla_lH_{ik}
-\lambda \nabla_l Q_{ik}\nabla_lH_{kj}) dx -\lambda \int_{\T^2} \nabla_j u_i(\Delta Q_{kj}H_{ik}-\Delta Q_{ik}H_{kj})dx\non\\
&& +\lambda \xi\int_{\T^2}\big(D\Delta{Q}+\Delta{Q}D):H dx
+ 4\lambda \xi \int_{\T^2} \nabla_lD_{ik}\nabla_lQ_{kj}H_{ij}\,dx\non\\
&& -2\lambda \xi\int_{\T^2} \Delta\big(Q_{kl}Q_{ji}\big)\nabla_ju_iH_{kl}\,dx
-4\lambda \xi\int_{\T^2}\nabla_m\big(Q_{kl}Q_{ji}\big)\nabla_m\nabla_ju_iH_{kl}\,dx\non\\
&& -\lambda \int_{\T^2}\frac{\partial{F(Q)}}{\partial Q} (u\cdot \nabla Q) : Hdx
+\lambda \int_{\T^2}\frac{\partial{F(Q)}}{\partial Q}S(\nabla u, Q): H dx\non\\
&& +\lambda \Gamma \int_{\T^2}\frac{\partial{F(Q)}}{\partial Q} H : H dx\non\\
&:=& \sum_{i=1}^{12} J_i.\label{DAT0}
\end{eqnarray}

Below we shall estimate the terms $J_1$ through $J_{12}$ in
\eqref{DAT0}. Let us take $\epsilon\in(0,1)$ to be a small constant
that will be determined later.

The term $J_1$ can be easily estimated by using the
Gagliardo-Nirenberg inequality \eqref{GN} and the lower-order estimate \eqref{lower order bound}:
\begin{align}
J_1 &\leq \|u\|_{L^4}\|\nabla{u}\|_{L^4}\|\Delta{u}\|_{L^2}\non\\
&\leq C\|u\|_{L^2}^{\frac12}\|\nabla{u}\|_{L^2}\|\Delta{u}\|_{L^2}^{\frac32} \non\\
&\leq \epsilon\|\Delta{u}\|_{L^2}^2+C\|\nabla{u}\|_{L^2}^4.\non
\end{align}
Recalling \eqref{def of H2d} and using again \eqref{lower order bound}, we observe that
\begin{align}
\|\Delta{Q}\|_{L^2}
&\leq \frac{1}{L}\|H(Q)\|_{L^2}+\frac{1}{L}\big\|aQ+c\tr(Q^2)Q\big\|_{L^2}
\non\\
&\leq \frac{1}{L}\|H(Q)\|_{L^2}+C(\|Q\|_{H^1}) \non\\
&\leq \frac{1}{L}\|H(Q)\|_{L^2}+C.\label{QQ}
\end{align}
Meanwhile, applying \eqref{GN1} and \eqref{AG} once more, we get
\begin{align*}
\|\nabla\Delta{Q}\|_{L^2}
&\leq
\frac{1}{L}\|\nabla{H}(Q)\|_{L^2}+\frac{1}{L}\Big\|a\nabla{Q}+c\nabla\big[\tr(Q^2)Q
\big]\Big\|_{L^2} \non\\
&\leq
\frac{1}{L}\|\nabla{H}(Q)\|_{L^2}+C\big(1+\|Q\|_{L^\infty}^2\big)\|\nabla{Q}\|_{L^2}
\non\\
&\leq \frac{1}{L}\|\nabla{H}(Q)\|_{L^2}+C\big(1+\|\Delta{Q}\|_{L^2}\big) \non\\
&\leq\frac{1}{L}\|\nabla{H}(Q)\|_{L^2}+C\big(1+\|\nabla{Q}\|_{L^2}^{\frac12}\|\nabla\Delta{Q}\|_{L^2}^{\frac12}+\|\nabla{Q}\|_{L^2}\big)\non\\
&\leq\frac{1}{L}\|\nabla{H}(Q)\|_{L^2}+\frac12\|\nabla\Delta{Q}\|_{L^2}+C,
\end{align*}
which implies
\begin{equation}
\|\nabla\Delta{Q}\|_{L^2} \leq\frac{2}{L}\|\nabla{H}(Q)\|_{L^2}+C.\label{QQQ}
\end{equation}
On the other hand, we infer from Agmon's inequality \eqref{AG} and the estimates \eqref{lower order bound}, \eqref{QQ} that
\begin{equation} \label{Argon}
\|Q\|_{L^\infty}\leq C\|Q\|_{L^2}^\frac12\|Q\|_{H^2}^\frac12\leq C(1+\|\Delta{Q}\|_{L^2}^{\frac12})\leq C(1+\|H(Q)\|_{L^2}^{\frac12}).
\end{equation}
As a consequence, we obtain from the H\"{o}lder inequality, the
Gagliardo-Nirenberg inequality \eqref{GN} and the Young's inequality
that
\begin{align}
J_2 &\leq C\|\nabla{u}\|_{L^2}\|Q\|_{W^{2,4}}\|H\|_{L^4} \non\\
&\leq C\|\nabla{u}\|_{L^2}\|Q\|_{H^2}^\frac12\|Q\|_{H^3}^\frac12\|H\|_{L^2}^\frac12\|H\|_{H^1}^\frac12\non\\
&\leq
C\|\nabla{u}\|_{L^2}\big(\|\Delta{Q}\|_{L^2}^{\frac12}\|\nabla\Delta{Q}\|_{L^2}^{\frac12}+\|\Delta{Q}\|_{L^2}+1\big)
\big(\|H\|_{L^2}^{\frac12}\|\nabla{H}\|_{L^2}^{\frac12}+\|H\|_{L^2}\big)\non\\
&\leq
C\|\nabla{u}\|_{L^2}\big(\|H\|_{L^2}^{\frac12}\|\nabla{H}\|_{L^2}^{\frac12}+\|\nabla{H}\|_{L^2}^\frac12+\|H\|_{L^2}+1 \big)\big(\|H\|_{L^2}^{\frac12}\|\nabla{H}\|_{L^2}^{\frac12}+\|H\|_{L^2}\big)\non\\
&\leq
\epsilon\|\nabla{H}\|_{L^2}^2+C\|\nabla{u}\|_{L^2}^2+C\|\nabla{u}\|_{L^2}^4+C\|H\|_{L^2}^2+C\|H\|_{L^2}^4\non\\
&\leq \epsilon\|\nabla{H}\|_{L^2}^2+C\A(1+\A).\non
\end{align}
For $J_3$, using the inequalities \eqref{GN} and \eqref{AG}, we obtain that
\begin{align}
J_3& \leq \|u\|_{L^4}\|\nabla Q\|_{L^4}(1+\|Q\|_{L^\infty}^2)\| H\|_{L^2}\non\\
&\leq C\|\nabla u\|_{L^2}\|\Delta Q\|_{L^2}^\frac12\|\nabla Q\|_{L^2}^\frac12(1+\|\Delta Q\|_{L^2})\|H\|_{L^2}\non\\
&\leq C\|\nabla u\|_{L^2} (1+\|\nabla \Delta Q\|_{L^2}^\frac12\|\nabla Q\|_{L^2}^\frac12)\|\Delta Q\|_{L^2}^\frac12\|H\|_{L^2}\non\\
&\leq \epsilon\|\nabla{H}\|_{L^2}^2 + C\|\nabla u\|_{L^2}^4+ C\|H\|_{L^2}^4+C\|H\|_{L^2}^2\non\\
&\leq \epsilon\|\nabla{H}\|_{L^2}^2+C\A(1+\A).\non
\end{align}
And terms $J_4$ and $J_5$ can be estimated as follows
\begin{align}
J_4&\leq C\|\nabla u\|_{L^4}\|\nabla Q\|_{L^4}\|\nabla H\|_{L^2}\non\\
&\leq C\|u\|_{L^2}^\frac12\|\Delta u\|_{L^2}^\frac12 \|\nabla Q\|_{L^2}^\frac12\|\Delta Q\|_{L^2}^\frac12\|\nabla H\|_{L^2}
\non\\
&\leq \epsilon\|\Delta u\|_{L^2}^2+\epsilon\|\nabla H\|_{L^2}^2+ C\|\nabla u\|_{L^2}^2+ C\|H\|_{L^2}^4\non\\
&\leq \epsilon\|\Delta u\|_{L^2}^2+\epsilon\|\nabla H\|_{L^2}^2+C\A(1+\A),\non
\end{align}
\begin{align}
J_5&\leq C\|\nabla u\|_{L^4}\|\Delta Q\|_{L^2}\|H\|_{L^4}\non\\
&\leq C\|\nabla u\|_{L^2}^\frac12\|\Delta u\|_{L^2}^\frac12(1+\|H\|_{L^2})(\|H\|_{L^2}^\frac12\|\nabla H\|_{L^2}^\frac12+\|H\|_{L^2})\non\\
&\leq \epsilon\|\Delta u\|_{L^2}^2+\epsilon\|\nabla H\|_{L^2}^2+ C\|\nabla{u}\|_{L^2}^2+C\|\nabla{u}\|_{L^2}^4+C\|H\|_{L^2}^2+C\|H\|_{L^2}^4\non\\
&\leq \epsilon\|\Delta u\|_{L^2}^2+\epsilon\|\nabla H\|_{L^2}^2+C\A(1+\A).\non
\end{align}
Besides, for $J_{10}$ and $J_{12}$ the following inequality holds
\begin{align}
& J_{10}+J_{12}\non\\
\leq &\ C\big(1+\|Q\|_{L^\infty}^2\big)\big(\|H\|_{L^2}+\|u\|_{L^\infty}\|\nabla{Q}\|_{L^2}\big) \|H\|_{L^2} \non\\
\leq &\
C(1+\|H\|_{L^2})\Big(\|H\|_{L^2}+\|u\|_{L^2}^{\frac12}\|\Delta{u}\|_{L^2}^{\frac12}\Big)\|H\|_{L^2}\non\\
\leq &\ C(1+\|H\|_{L^2})\|H\|_{L^2}^2+C\|H\|_{L^2}^2\|u\|_{L^2}^{\frac12}\|\Delta{u}\|_{L^2}^{\frac12}
+C\|H\|_{L^2}\|\nabla{u}\|_{L^2}^{\frac12}\|\Delta{u}\|_{L^2}^{\frac12}\non\\
\leq&\ \epsilon\|\Delta{u}\|_{L^2}^2+C\|\nabla{u}\|_{L^2}^2+C\|\nabla{u}\|_{L^2}^4+C\|H\|_{L^2}^2+C\|H\|_{L^2}^4\non\\
\leq &\ \epsilon\|\Delta u\|_{L^2}^2+C\A(1+\A).\non
\end{align}

It remains to estimate the terms $J_6, ..., J_{9}$ and $J_{11}$ involving the parameter $\xi$, which all vanish when $\xi=0$.
Thus, we only need to consider the case $\xi\neq 0$ (with $\xi$ being fixed).

The term $J_6$ can be estimated in the same way as for $J_2$, that is
\begin{align}
J_6&\leq C|\xi|\|\nabla u\|_{L^2}\|\Delta Q\|_{L^4}\|H\|_{L^4}\non\\
&\leq \epsilon\|\nabla{H}\|_{L^2}^2+C(1+|\xi|^2)\A(1+\A).\non
\end{align}
For $J_7$, using the H\"older inequality, \eqref{GN}, \eqref{QQ} and Young's inequality, we have
\begin{align}
J_7 &\leq C|\xi|\|\Delta u\|_{L^2}\|\nabla Q\|_{L^4}\|H\|_{L^4}\non\\
&\leq C|\xi|\|\Delta u\|_{L^2}\|\nabla Q\|_{L^2}^\frac12\|\Delta Q\|_{L^2}^\frac12\big(\|H\|_{L^2}^\frac12\|\nabla H\|_{L^2}^\frac12+\|H\|_{L^2}\big)\non\\
&\leq \epsilon\|\Delta u\|_{L^2}^2+\epsilon\|\nabla H\|_{L^2}^2+C(1+|\xi|^4)(\|H\|_{L^2}^2+\|H\|_{L^2}^4)\non\\
&\leq \epsilon\|\Delta u\|_{L^2}^2+\epsilon\|\nabla H\|_{L^2}^2+C(1+|\xi|^4)\A(1+\A).\non
\end{align}
Next, we first treat $J_{11}$ and postpone the estimates for terms $J_8$, $J_9$ that are more involved.
\begin{align}
J_{11}
&\leq C\big(1+\|Q\|_{L^\infty}^2\big)\|S(\nabla{u},Q)\|_{L^2}\|H\|_{L^2} \non\\
&\leq
C(1+\|H\|_{L^2})\big(1+\|Q\|_{L^\infty}\big)^2\|\nabla{u}\|_{L^2}\|H\|_{L^2}\non\\
&\leq C(1+\|H\|_{L^2})^2\|\nabla{u}\|_{L^2}\|H\|_{L^2}\non\\
&\leq C\|\nabla{u}\|_{L^2}\|H\|_{L^2}+ C\|H\|_{L^2}^3\|\Delta{u}\|_{L^2}^\frac12\|u\|_{L^2}^\frac12\non\\
&\leq \epsilon \|\Delta{u}\|_{L^2}^2+C\|\nabla{u}\|_{L^2}^2+C\|H\|_{L^2}^2+C\|H\|_{L^2}^4\non\\
&\leq \epsilon \|\Delta{u}\|_{L^2}^2+ C\A(1+\A).\non
\end{align}
Now, let us consider the term $J_8$. By a similar argument for $J_5$ and using the Br\'{e}zis-Gallouet inequality \eqref{BG},
we obtain that
\begin{align}
J_8 &\leq 2|\xi|\|\nabla{u}\|_{L^4}\|H\|_{L^4}\|\Delta(QQ)\|_{L^2} \non\\
&\leq
C|\xi|\|\nabla{u}\|_{L^4}\|H\|_{L^4}\big(\|Q\|_{L^\infty}\|\Delta{Q}\|_{L^2}+\|\nabla{Q}\|_{L^4}^2  \big)\non\\
&\leq
C|\xi|\|\nabla{u}\|_{L^4}\|H\|_{L^4}\big(\|Q\|_{L^\infty}\|\Delta{Q}\|_{L^2}+\|\Delta Q\|_{L^2}\|\nabla Q\|_{L^2} \big)\non\\
&\leq
C|\xi|\|\nabla u\|_{L^2}^\frac12\|\Delta u\|_{L^2}^\frac12(1+\|H\|_{L^2})(\|H\|_{L^2}^\frac12\|\nabla H\|_{L^2}^\frac12+\|H\|_{L^2})\big(\|Q\|_{L^\infty}+1 \big) \non\\
&\leq \epsilon \|\Delta u\|^2+\epsilon \|\nabla H\|^2\non\\
&\quad +C(|\xi|+|\xi|^4)(1+\|Q\|_{L^\infty}^2)(\|\nabla u\|_{L^2}^4+\|\nabla u\|_{L^2}^2+\|H\|_{L^2}^4+\|H\|_{L^2}^2)\non\\
&\leq  \epsilon \|\Delta u\|^2+\epsilon \|\nabla H\|^2 +C(|\xi|+|\xi|^4)\mathcal{B}\A(1+\A),\non
\end{align}
where we have set
\begin{equation}
\mathcal{B}=e+\ln(e+\A)>e.\label{BB}
\end{equation}

Concerning the last term $J_9$, by the H\"older inequality we have, for any $p\in (0,1)$,
\begin{align}
J_9 &\leq
4|\xi|\|Q\|_{L^\infty}\|\nabla{Q}\|_{L^\frac2p}\|\nabla^2u\|_{L^2}\|H\|_{L^\frac{2}{1-p}}.\label{J9}
\end{align}
For any $p\in (0,1/2)$, applying the $L^p$-interpolation inequality Lemma \ref{Lp}, with $\eta=p^{-1}>1$ and $\eta=(1-p)^{-1}\in (1,2)$, respectively, we deduce that
\begin{align}
\|\nabla Q\|_{L^\frac2p}
&\leq C\sqrt{p^{-1}}\|Q\|_{H^2}^{1-p}\|\nabla Q\|_{L^2}^p\non\\
&\leq C\sqrt{p^{-1}}\|H\|_{L^2}^{1-p}\|\nabla Q\|_{L^2}^p+ C\sqrt{p^{-1}}\non\\
&\leq C\sqrt{p^{-1}}\|H\|_{L^2}^{1-p}+ C\sqrt{p^{-1}},
\end{align}
and
\begin{align}
\|H\|_{L^\frac{2}{1-p}}
&\leq C\sqrt{(1-p)^{-1}}\|H\|_{H^1}^{p}\|H\|_{L^2}^{1-p}\non\\
&\leq C\sqrt{(1-p)^{-1}}\|\nabla H\|_{L^2}^{p}\|H\|_{L^2}^{1-p}+ C\sqrt{(1-p)^{-1}}\|H\|_{L^2}\non\\
&\leq C\|\nabla H\|_{L^2}^{p}\|H\|_{L^2}^{1-p}+ C\|H\|_{L^2}.\label{eH}
\end{align}
Hence, by the Br\'{e}zis-Gallouet inequality \eqref{BG}, estimates \eqref{J9}-\eqref{eH} and the Young's inequality, we have
\begin{align}
J_9&\leq C|\xi|\sqrt{p^{-1}}\|\Delta u\|_{L^2}\|Q\|_{L^\infty} \non\\
&\quad \times \big(\|\nabla H\|_{L^2}^p\|H\|_{L^2}^{2(1-p)}+\|H\|_{L^2}^{2-p}
+\|\nabla H\|_{L^2}^p\|H\|_{L^2}^{1-p}+\|H\|_{L^2}\big)\non\\
&\leq \epsilon|\xi|\|\Delta u\|_{L^2}^2+\epsilon|\xi|\|\nabla H\|_{L^2}^2
+C|\xi|p^{-\frac{1}{1-p}}\|Q\|_{L^\infty}^\frac{2}{1-p}(\|H\|_{L^2}^4+\|H\|_{L^2}^2)\non\\
&\quad + C|\xi|p^{-1}\|Q\|^2_{L^\infty}(\|H\|_{L^2}^4+\|H\|_{L^2}^2)\non\\
&\leq \epsilon|\xi|\|\Delta u\|_{L^2}^2+\epsilon|\xi|\|\nabla H\|_{L^2}^2
+C|\xi|p^{-\frac{1}{1-p}}[e+\ln(e+\A)]^\frac{1}{1-p}(\|H\|_{L^2}^4+\|H\|_{L^2}^2)\non\\
&\quad + C|\xi|p^{-1}[e+\ln(e+\A)](\|H\|_{L^2}^4+\|H\|_{L^2}^2)\non\\
&:=\epsilon|\xi|\|\Delta u\|_{L^2}^2+\epsilon|\xi|\|\nabla H\|_{L^2}^2+J_{9a}+J_{9b},\quad \forall\, p\in\big(0, \frac12\big).
\label{J92}
\end{align}
Since the constant $C$ in the estimate \eqref{J92} is independent of the parameter $p\in (0, 1/2)$, then, in the spirit of \cite{PZ11},
we can take the exponent
$$p=(1+\ln \mathcal{B})^{-1},$$ where $\mathcal{B}$ is given in \eqref{BB}. We note that
with this choice $p$ may not be a constant, but it is always true that
$p \in (0, 1/2)$. Then it follows from \eqref{J92} that
\begin{align}
J_{9b}&\leq C|\xi| (1+\ln \mathcal{B})\mathcal{B}\A(1+\A),
\end{align}
and
\begin{align}
J_{9a}&\leq C|\xi|p^{-\frac{1}{1-p}}[e+\ln(e+\A)]^\frac{1}{1-p}(\|H\|_{L^2}^4+\|H\|_{L^2}^2)\non\\
&\leq C|\xi| [(1+\ln \mathcal{B})\mathcal{B}]^{1+\frac{1}{\ln B}} \A(1+\A)\non\\
&\leq C|\xi| (1+\ln \mathcal{B})\mathcal{B}\A(1+\A),\label{J9a}
\end{align}
where we have used the following simple fact such that the quantity
$$
[(1+\ln \mathcal{B})\mathcal{B}]^\frac{1}{\ln B}=[(1+\ln \mathcal{B})e^{\ln \mathcal{B}}]^\frac{1}{\ln \mathcal{B}}
=e(1+\ln \mathcal{B})^\frac{1}{\ln \mathcal{B}}
$$
is uniformly bounded for all $\mathcal{B}>e$.
As a consequence of \eqref{J92}-\eqref{J9a}, we deduce that
\begin{align}
J_9&\leq \epsilon|\xi|\|\Delta u\|_{L^2}^2+\epsilon|\xi|\|\nabla H\|_{L^2}^2+C|\xi| (1+\ln \mathcal{B})\mathcal{B}\A(1+\A).\label{J93}
\end{align}

Now we take the small constant $$\epsilon\in\Big(0, \ \ \frac{\min\{\nu, \lambda\Gamma\}}{2(7+|\xi|)}\Big).$$
From \eqref{DAT0} and the above estimates for terms $J_1$,... $J_{12}$, it follows that
\begin{align}\label{dA1}
&\frac{d\A}{dt} + \frac{\nu}{2}\|\Delta{u}\|_{L^2}^2+ \frac{\lambda\Gamma}{2}\|\nabla{H}\|_{L^2}^2
\leq C_\ast\big[1+|\xi|(1+\ln \mathcal{B})\mathcal{B}\big]\A(1+\A),
\end{align}
which easily implies the conclusion \eqref{hiA}. The proof is complete.
\end{proof}

\begin{remark}\label{high2Dxi0}
If $\xi=0$, the inequality \eqref{hiA} reduces to
\begin{align}
&\frac{d}{dt}\A(t) \leq  C_*[e+\A(t)]\A(t),\label{hiAxi0}
\end{align}
which is the same as the higher-order energy inequality derived in \cite[Lemma 7.1]{ADL15}.
\end{remark}

\section{Global Strong Solutions in $2D$}
\setcounter{equation}{0}

In this section, we show that starting from initial data with higher regularity, the problem \eqref{navier-stokes}-\eqref{IC} admits a unique global strong solution.

\subsection{Semi-Galerkin approximation scheme}

We can work with a semi-Galerkin scheme in the periodic setting, which
is similar to \cite{LL95} for the simplified Ericksen-Leslie system
for incompressible nematic liquid crystal flow. For the convenience
of the readers, we briefly describe it below. Recalling the classical
spectral theorem for compact operators in Hilbert spaces and
standard results for the stationary Stokes system, we have the
following results on eigenfunctions of the Stokes operator
$\mathcal{S}$ for $u$. Let $\{ v_n \}_{n=1}^\infty$ be the
eigenvectors of the Stokes operator $\mathcal{S}$ in the torus
$\T^2$ with zero mean,
 \begin{align}
 &\mathcal{S} v_n= \kappa_n v_n,\quad \nabla\cdot v_n = 0, \quad \int_{\T^2}v_n(x)\,dx=0, \ \ \mbox{in } \T^2,\non\\
 & v_n(x+e_i)=v_n(x),\quad x\in \T^2,\non
 \end{align}
 where $0<\kappa_1\leq\kappa_2\leq ...\nearrow +\infty$ are eigenvalues. The eigenvectors
 ${v_n}$ are smooth and the sequence $\{v_n\}_{n=1}^\infty$ forms an orthogonal
basis of $\mathbf{H}$ as well as $\mathbf{V}$ (see e.g., \cite{Te}).

Taking an arbitrary but fixed integer $N \in \mathbb{N}$, we
consider the finite-dimensional space $V_N = {\rm
span}\{v_n\}_{n=1}^N$ along with the orthogonal projection operators
$\Pi_{N}: \mathbf{H}\to V_N$, which are bounded linear operators
with norms bounded by one. For any $T>0$, we seek approximations of
solutions to the problem \eqref{navier-stokes}-\eqref{IC}. The
approximation of velocity $u^N$ takes the form
$$u^{N}=\sum_{i=1}^N h_i(t)v_i(x),$$
which solves
\begin{align}
&\int_{\T^2} (u^{N})_t\cdot v_k\,dxdt-\int_{\T^2} \left(u^{N}\otimes  u^{N}\right) : \nabla v_k dx + \nu \int_{\T^2} \nabla u^{N}: \nabla v_k\,dx\nonumber\\
&\quad =-\int_{\T^2} (\sigma^{N}+\tau^{N}):\nabla v_k\,dx,\quad \forall\,t \in (0,T), \label{approx navier-stokes}
\end{align}
for any $k=1,...,N$. In \eqref{approx navier-stokes}, the approximations of stress tensors are given by
\begin{eqnarray}
\tau^{N}&\defeq&-\xi\big(Q^{N}+\frac{1}{2}\Id\big) H^{N}(Q^N)-\xi H^{N}(Q^N)\big(Q^{N}+\frac{1}{2}\Id\big)\nonumber\\
&& +2\xi\big(Q^{N}+\frac{1}{2}\Id\big)\tr(Q^{N}H^{N}(Q^N))-L\nabla{Q}^{N}\odot\nabla{Q}^{N},\label{symmetic tensor appa}\\
\sigma^{N}&\defeq& Q^{N}H^{N}(Q^N)-H^{N}(Q^N)Q^{N},\label{skew-symmetric tensor appa}
\end{eqnarray}
where
\begin{align}
H^{N}(Q^N)&\defeq L\Delta{Q}^{N}-aQ^{N}-cQ^{N}\tr((Q^{N})^2),\label{def of H appa}
\end{align}
On the other hand, the approximate function $Q^{N}$ is determined in terms of $u^{N}$ as the unique solution of the parabolic system
\begin{align}
Q^{N}_t+u^{N}\cdot\nabla Q^{N}-S(\nabla{u}^{N}, Q^{N})&=\Gamma{H}^{N}(Q^N),\quad (x,t)\in \T^{2} \times \mathbb{R}^+, \label{approx Q equ}
\end{align}
 where
\begin{align}
S^N(\nabla{u}^{N}, Q^{N})&\defeq (\xi D^{N}+\Omega^N)\big(Q^{N}+\frac12 \Id
\big)+\big(Q^{N}+\frac12 \Id \big)(\xi D^{N}-\Omega^{N})\nonumber\\
&\quad -2\xi\big(Q^{N}+\frac12 \Id\big)\tr(Q^{N}\nabla{u}^{N})\label{S1 appa}
\end{align}
with
$$D^{N}=\frac{\nabla u^{N}+\nabla^Tu^{N}}{2},\quad \Omega^{N}=\frac{\nabla
u^{N}-\nabla^Tu^{N}}{2}.$$
The initial conditions are given by
\begin{equation}
u^{N}|_{t=0}= \Pi_{N} u_0, \quad Q^{N}|_{t=0}= Q_{0}, \quad x\in \T^2. \label{approx IC}
\end{equation}

\subsection{Proof of Theorem \ref{strong2d}}

\noindent The proof for the existence of global strong solutions
consists of several steps. \medskip

\textbf{Existence of approximate solutions}. For any fixed integer $N$, we have the following result on local existence of the approximate solution $(u^{N}, Q^{N})$:

\begin{proposition} \label{Theorem for the approx prob}
Suppose $u_0 \in\mathbf{V}$, $Q_0 \in H^2(\T^2, S_0^{(2)})$. For any $N\in\mathbb{N}$, there exists $T_N
> 0$ depending on $\|u_0\|_{H^1}$, $\|Q_0\|_{H^2}$ and $N$ such that the
approximate problem \eqref{approx navier-stokes}-\eqref{approx IC}
admits a solution $(u^N, Q^N)$ satisfying
\begin{align*}
 & u^N \in L^{\infty}(0, T_N; \mathbf{V}) \cap L^2(0, T_N; H^2(\T^2,\RR^2))\cap H^1(0,T_N; \mathbf{H}),\\
 & Q^N \in L^{\infty}(0, T_N; H^2(\T^2,S^{(2)}_0))\cap L^2(0, T_N; H^3(\T^2,S^{(2)}_0))\cap H^1(0,T_N; H^1(\T^2,S^{(2)}_0)).
 \end{align*}
\end{proposition}
Proposition \ref{Theorem for the approx prob} can be proved by a classical Schauder's argument (see e.g., \cite{LL95}).
Indeed, given a vector $\tilde{u} \in C([0,T]; V_N)$, then we find a $Q= Q[\tilde{u}]$ by solving the equation \eqref{approx Q equ} with $u^N$ replaced by $\tilde{u}$. Inserting $Q[\tilde{u}]$ back into the equation \eqref{approx navier-stokes}, then the solution $u$ to the resulting ODE system defines a mapping $\mathcal{T}: \tilde{u} \mapsto {\cal T}[\tilde{u}]=u$. It is standard to show that ${\cal T}$ admits a fixed point by means of the classical Schauder's argument on $(0,T_N)$, with certain $T_N>0$ depending on $\|u_0\|_{H^1}$, $\|Q_0\|_{H^2}$ and $N$.  We leave the detailed proof to interested readers.
\begin{remark}
Since we are working in the periodic domain $\T^2$, by the classical regularity theory for parabolic equations (cf. \cite{LA1}) and a bootstrap argument,
we can see that $(u^N, Q^N)$ is $C^\infty$ in the interior of $\T^2\times (0, T_N)$.
\end{remark}

In order to prove the existence of global strong solutions, we need to derive some uniform estimates for approximate solutions $(u^N, Q^N)$
that are independent of the approximation parameter $N$ as well as the time $t$.

\medskip
\textbf{Lower-order estimates}.  A similar argument like in \cite{ADL14} yields that the approximate solutions satisfy the following energy identity
\begin{equation}
\frac{d}{dt}\E^N(t)=-\nu\int_{\T^2} |\nabla{u}^N|^2\,dx-\lambda \Gamma\int_{\T^2} |H^N(Q^N)|^2\,dx\leq 0,\quad \forall\, t\in[0,T_N).\label{BELN}
\end{equation}
where
\begin{equation}\label{total energyN}
\E^N(t)\defeq \frac12\int_{\T^2} |u^N|^2(x,t)\,dx+\lambda \F(Q^N(t)).\non\\
\end{equation}
As in Lemma \ref{lower order bound}, the energy identity \eqref{BELN} provides uniform estimates for $u^N$ and $Q^N$ such that
\begin{align}
&\|u^N(t)\|_{L^2}+\|Q^N(t)\|_{H^1}\leq C, \qquad \forall\, t\in[0,T_N),\label{lowN1}\\
&\int_0^t\int_{\T^d} |\nabla{u}^N(\tau)|^2+ |\Delta{Q}^N(\tau)|^2\,dxd\tau < C(1+t),\quad \forall\, t\in[0,T_N),\label{lowN2}
\end{align}
where the constant $C>0$ depends on $\|u_0\|, \|Q_0\|_{H^1}, L, \lambda, \nu, \Gamma, a, c$ and
$\T^2$, but it is independent of the parameter $N$ and the time $t$.

\medskip
\textbf{Higher-order estimates}. It is easy to see that the
calculations we made in Section 3 for smooth solutions $(u, Q)$ to the
problem \eqref{navier-stokes}-\eqref{IC} still hold for the
approximate solutions $(u^N, Q^N)$. Thus, for  $(u^N, Q^N)$, we introduce the quantity
\begin{equation} \label{def of AN}
\A_N(t)=\|\nabla{u}^N(t)\|_{L^2}^2+\lambda\|H^N(Q^N(t))\|_{L^2}^2.
\end{equation}
In particular, we infer from \eqref{BELN} that
\begin{equation} \label{K0a}
\int_0^t\int_{\T^2} \nu |\nabla{u}^N(\tau)|^2+\lambda \Gamma |H^N(Q^N(\tau))|^2\,dxd\tau \leq K, \quad \forall\, t\in[0,T_N),
\end{equation}
where $K>0$ is a constant that only depends on $\|u_0\|_{L^2}, \|Q_0\|_{H^1}, L, \lambda, a, c, \T^2$. Then we have
\begin{equation}
\int_0^{t}\A_N(\tau)\,d\tau \leq \frac{K}{\min\{\nu,  \Gamma\}},\quad \forall\, t\in[0,T_N).\label{intAN}
\end{equation}

On the other hand, using the lower-order estimate \eqref{lowN1}, for any $\xi\in \RR$, we can get the following higher-order energy inequality for all $ t\in [0,T_N)$:
\begin{eqnarray}
&&\frac{d}{dt}\A_N(t) + \frac{\nu}{2}\|\Delta{u}^N\|_{L^2}^2+ \frac{\lambda\Gamma}{2}\|\nabla{H}^N(Q^N)\|_{L^2}^2\non\\
&\leq& C_\ast\big[1+|\xi|\big[1+\ln(e+\ln(e+\A_N(t)))\big](e+\ln(e+\A_N(t)))\big][e+\A_N(t)]\A_N(t),\label{hiAN}
\end{eqnarray}
where $C_\ast>0$ is a constant that depends on $\nu, \Gamma, L, \lambda, a, c,
\T^2, \|u_0\|_{L^2}, \|Q_0\|_{H^1}$, $\xi$, but is independent of $N$ and $t$.

Now we consider two cases.

\textbf{Case 1}. If $\xi=0$, then we infer from \eqref{hiAN} that
\begin{align}
\frac{d}{dt}\ln[e+\A_N(t)]\leq C_\ast\A_N(t),\non
\end{align}
which implies
\begin{align}
\A_N(t) \leq [e+\A_N(0)] e^{C_*\int_0^t \A_N(\tau) d\tau}\leq [e+\A_N(0)]e^{\frac{C_*K}{\min\{\nu,  \Gamma\}}} \leq C, \quad \forall\, t\in [0,T_N),\non
\end{align}
where
$C>0$ is a constant that depends on $\nu, \Gamma, L, \lambda, a, c, \T^2, \|u_0\|_{\mathbf{V}}, \|Q_0\|_{H^2}$.

\textbf{Case 2}. If $\xi\neq 0$, then we deduce from \eqref{hiAN} that
\begin{align}
\frac{d}{dt}\ln \mathcal{Z}_N(t)\leq C_*(1+|\xi|)\A_N(t),\label{dZ}
\end{align}
where
$$\mathcal{Z}_N(t)=1+\ln[1+\ln[e+\A_N(t)]].$$
Integrating \eqref{dZ} with respect to time and using \eqref{intAN}, we have
\begin{align}
\ln \mathcal{Z}_N(t)
&\leq \ln \mathcal{Z}_N(0)+C_*(1+|\xi|)\int_0^t\A_N(\tau)d\tau\non\\
& \leq \ln \mathcal{Z}_N(0)+\frac{(C_*)^2 K(1+|\xi|)}{\min\{\nu,  \Gamma\}},\qquad \forall\, t\in [0,T_N),\non
\end{align}
which again yields that
\begin{align}
\A_N(t) \leq C, \quad \forall\, t\in [0,T_N).
\end{align}

For both cases, after integrating the differential inequality \eqref{hiAN} with respect to time, we obtain that
\begin{equation}
\int_0^t \big(\|\Delta{u}^N(\tau)\|_{L^2}^2+ \|\nabla{H}^N(Q^N(\tau))\|_{L^2}^2\big)d\tau \leq C, \quad \forall\, t\in [0,T_N).\non
\end{equation}
As a consequence, we have the following uniform higher-order  estimates:
\begin{align}
&\|u^N(t)\|_{H^1}+\|Q^N(t)\|_{H^2}\leq C, \qquad \forall\, t\in[0,T_N),\label{highN1}\\
&\int_0^t\int_{\T^2} |\Delta{u}^N(\tau)|^2+ |\nabla
\Delta{Q}^N(\tau)|^2\,dxd\tau < C(1+t),\quad \forall\, t\in
[0,T_N),\label{highN2}
\end{align}
where the constant $C>0$ depends on $\|u_0\|_{H^1}, \|Q_0\|_{H^2}, L, \lambda, \nu, \Gamma, a, c, \xi$ and
$\T^2$, but it is independent of the parameter $N$ and the time $t$.

\medskip
\textbf{Passage to the limit $N\to \infty$}. First, we can deduce
from the above uniform-in-time lower-order and higher-order
estimates \eqref{lowN1}, \eqref{highN1} that the approximate
solutions $(u^N, Q^N)$ can not blow up in finite time. Thus, for any
$N\in \mathbb{N}$, it holds $T_N=+\infty$ such that every
approximate solution $(u^N, Q^N)$ can be extended to the time
interval $[0,T]$ for arbitrary $T>0$.

Second, since the uniform estimates \eqref{lowN1}, \eqref{lowN2},
\eqref{highN1}, \eqref{highN2} are also independent of the
approximation parameter $N$, we infer from the equations
\eqref{approx navier-stokes}, \eqref{approx Q equ} and the H\"older
inequality that for any $T>0$ and $N\in \mathbb{N}$,
\begin{align*}
 & u^N \in L^{\infty}(0, T; \mathbf{V}) \cap L^2(0, T; H^2(\T^2,\RR^2))\cap H^1(0,T; \mathbf{H}),\\
 & Q^N \in L^{\infty}(0, T; H^2(\T^2,S^{(2)}_0))\cap L^2(0, T; H^3(\T^2,S^{(2)}_0))\cap H^1(0,T; H^1(\T^2,S^{(2)}_0)).
 \end{align*}
The above uniform estimates together with standard weak compactness results and the Aubin-Lions compactness lemma (see e.g., \cite[Cor. 4, Sec. 8]{SI}) enable us to pass to the limit as $N\rightarrow \infty$ (up to a subsequence) to obtain a limit pair $(u, Q)$, which turns
out to be a global strong solution to the original Navier-Stokes and Q-tensor system \eqref{navier-stokes}-\eqref{IC}. Since the argument is standard (cf. \cite{ADL15}), we omit the details here.

\medskip

\textbf{Uniqueness}. The uniqueness of strong solutions is a direct consequence of \cite[Section 5]{PZ11}, where a weak-strong uniqueness
result is given in $\RR^2$.

Let $(u_i, Q_i)$, $i=1,2$ be two global strong solutions of the problem \eqref{navier-stokes}-\eqref{IC}
subject to initial data $(u_{i0}, Q_{i0})$, $i=1,2$, respectively. Since we are dealing with the periodic domain, using the same argument
as in \cite{PZ11}, we can obtain the following estimates (however, without any smallness assumption on $\xi$):
\begin{align}
& \frac{d}{dt}\Big(\|u_1-u_2\|_{L^2}^2+\lambda L\|\nabla(Q_1-Q_2)\|_{L^2}^2+\lambda\|Q_1-Q_2\|_{L^2}^2\Big)\non\\
& \ \ +\nu\|\nabla(u_1-u_2)\|_{L^2}^2+\lambda \Gamma L^2\|\Delta (Q_1-Q_2)\|_{L^2}^2 \non\\
\leq \ & h(t)\Big(\|u_1-u_2\|_{L^2}^2+\lambda L\|\nabla(Q_1-Q_2)\|_{L^2}^2+\lambda\|Q_1-Q_2\|_{L^2}^2\Big),
\end{align}
where $h(t)\in L^1(0,T)$ is a time-integrable function. As a
consequence, we have
\begin{eqnarray}
&& \| (u_1-u_2)(t)\|_{L^2}^2+\|(Q_1-Q_2)(t)\|_{H^1}^2 \non\\
&& \ \ +\int_0^t(\|\nabla
 (u_1-u_2)(s)\|_{L^2}^2+\|\Delta(Q_1-Q_2)(s)\|_{L^2}^2) ds \non\\
 &\leq&  Ce^{\int_0^t h(s)ds}\left(\|u_{01}-u_{02}\|_{L^2}^2+\|Q_{01}-Q_{02}\|_{H^1}^2\right),\quad \forall\, t\in(0,T).\label{continuous}
\end{eqnarray}
Therefore, the global strong solution to the problem
\eqref{navier-stokes}-\eqref{IC} is unique.

The proof of Theorem \ref{strong2d} is complete.

\begin{remark}
It seems impossible to prove any continuous dependence results on
initial data for the strong solutions obtained above in the space $\mathbf{V}\times
H^2$. Nevertheless, the estimate \eqref{continuous} implies that for any $(u_0,
Q_0)\in \mathbf{V}\times H^2$, we are able to define a \emph{closed
semigroup} $\Sigma(t)$ for  $t\geq 0$ (in the sense of \cite{PZ07})
by setting $(u(t),Q(t))=\Sigma(t)(u_0, Q_0)$,  where $(u,Q)$ is the global
strong solution to the problem \eqref{navier-stokes}-\eqref{IC}.
\end{remark}

\section{Long-time Behavior}
\setcounter{equation}{0}

In this section we investigate the long-time behavior of the global
strong solution to problem \eqref{navier-stokes}-\eqref{IC}
established in Theorem \ref{strong2d}. The related study consists of
two steps. First, we prove that the asymptotic limit point of the global
strong solution $(u(t), Q(t))$ as $t$ tends to infinity is unique.
Then we provide an uniform estimate of the convergence rate.

\subsection{Characterization of $\omega$-limit set}

For any initial datum $(u_0, Q_0)\in \mathbf{V}\times H^2\big(\T^2, S_0^{(2)}\big)$,
we denote its $\omega$-limit set by
\begin{align*}
\omega(u_0, Q_0)\defeq\big\{(u_\infty, Q_\infty)| \;\exists
\;\{t_n\} \nearrow \infty\,:\, u(t_n)\rightarrow
u_\infty \mbox{ in } L^2, Q(t_n)\rightarrow Q_\infty \mbox{ in } H^1
\mbox{ as } n\rightarrow \infty \big\}.
\end{align*}
On the other hand, we denote the set of solutions to the elliptic problem
$$
  \mathfrak{S} = \big\{Q_*: L\Delta{Q}_*-aQ_*-c\tr(Q_*^2)Q_*=0,\ \ Q_*\in S_0^{(2)}\ \text{and}\  Q_*(x+e_i)=Q_*(x) \mbox{ in } \T^2\big\}.
$$
\begin{remark}
Since the free energy $\mathcal{F}(Q)$ given by \eqref{elastic energy} is bounded from below, using the classical variational method and the elliptic regularity theorem, it is easy to see that the set $\mathfrak{S}$ is non-empty. Besides, every $Q_*\in \mathfrak{S}$ is a critical point of $\mathcal{F}(Q)$.
\end{remark}
Next, by virtue of the properties of the $\omega$-limit set
$\omega(u_0, Q_0)$ as well as the higher-order energy term $\A(t)$, we have
\begin{lemma} \label{lemma on decay property}
Suppose that the assumptions in Theorem \ref{strong2d} are
satisfied. For any initial datum $(u_0, Q_0)\in \mathbf{V}\times H^2\big(\T^2, S_0^{(2)}\big)$, we have $\omega(u_0, Q_0)$ is
a nonempty bounded subset in $\mathbf{V}\times H^2\big(\T^2, S_0^{(2)}\big)$ which satisfies
$$
  \omega(u_0, Q_0) \in \big\{(0, Q_*): Q_*\in \mathfrak{S}\big\}
$$
and the total energy $\E(t)$ is a constant on $\omega(u_0, Q_0)$. Besides, the unique global strong solution $(u, Q)$ has the following decay property:
\begin{equation}
\displaystyle\lim_{t\rightarrow +\infty}\big(\|u(t)\|_{H^1}+\|H(Q(t))\|_{L^2}\big)=0.\label{decayu}
\end{equation}
\end{lemma}
\begin{proof}
Since the global strong solution $(u, Q)$ obtained in Theorem
\ref{strong2d} satisfies the higher-order energy inequality
\eqref{hiA}, using a similar argument as in Section 4.2, we get
\begin{equation}
\mathcal{A}(t)\leq C,\quad \forall\, t\geq 0,\label{esAA}
\end{equation}
where $C>0$ depends on $\|u_0\|_{H^1}, \|Q_0\|_{H^2}, L, \lambda, \nu, \Gamma, a, c, \xi$ and $\T^2$.
As a consequence, it follows from \eqref{hiA} and \eqref{esAA} that
\begin{equation}
\frac{d}{dt}\A(t) \leq C, \quad \forall\, t\geq 0.\label{lipA}
\end{equation}
On the other hand, the energy identity \eqref{es1} for $(u, Q)$ yields that
\begin{equation} \label{K0}
\int_0^{+\infty} \int_{\T^2} \nu |\nabla{u}|^2+\lambda \Gamma |H(Q)|^2\,dxdt \leq K_0,
\end{equation}
where $K_0>0$ is a constant that only depends on $\|u_0\|_{L^2}, \|Q_0\|_{H^1}, L, \lambda, a, c, \T^2$.
This implies that $\int_0^{+\infty} \A(t)\,dt <+\infty$, which together with \eqref{lipA} leads to the conclusion  \eqref{decayu}.

Since the total energy $\E(t)$ is non-increasing in time and bounded from below by a generic constant, there exists a constant $\F_\infty\in \RR$ such that
\begin{equation}
\lim_{t\to+\infty}\mathcal{E}(t)=\F_\infty.\label{limE}
\end{equation}
By the definition of $\omega(u_0, Q_0)$, it is easy to see that $\E(t)$ is equal to the constant $\F_{\infty}$ on the set $\omega(u_0, Q_0)$. The proof is complete.
\end{proof}

\subsection{Convergence to equilibrium}

In general, we cannot directly conclude that each global strong solution of
system \eqref{navier-stokes}-\eqref{IC} converges to a single
equilibrium as $t\to +\infty$ because the set of steady states
$\mathfrak{S}$ for $Q$-tensors can have a complicated structure.
Besides, since we are working in the periodic torus $\T^2$, we may
expect the dimension of the set $\mathfrak{S}$ to be at least $2$.
However, we may establish a gradient inequality of \L ojasiewicz-Simon type for this matrix valued function
$Q$ and apply Simon's idea (see \cite{FS,S83}) to accomplish our goal.

To begin with, using \eqref{lower order bound} and \eqref{esAA}, we
have the following uniform-in-time estimates
\begin{equation}
\|u(t)\|_{H^1}+\|Q(t)\|_{H^2} \leq C, \quad \forall\, t\geq 0.\label{uniES}
\end{equation}
Then, from Lemma \ref{lemma on decay property} we infer that there exists an
increasing unbounded sequence $\{t_n\}_{n\in N}$ and a matrix function $Q_\infty\in
H^2\big(\T^2, S_0^{(2)}\big)$, such that
\begin{equation} \label{sequential convergence}
 \displaystyle\lim_{t_n\rightarrow +\infty}\big\|Q(t_n)-Q_\infty\big\|_{H^1}=0,
\end{equation}
where $(0, Q_\infty) \in \omega(u_0, Q_0)$.

We now proceed to prove the convergence of $Q(t)$ to $Q_\infty$ for all time as $t\to+\infty$, which implies that the
$\omega$-limit set $\omega(u_0, Q_0)$ is actually a singleton. For this purpose, the following \L ojaciewicz-Simon type
inequality plays an important role.
\begin{lemma}\label{lemma on LS inequality}
 Let $Q_\ast\in H^1(\T^2, S_0^{(2)})$ be a critical point of the energy functional $\F(Q)$. Then there exist some
constants $\theta\in (0, \frac12)$ and $\beta>0$ depending on $Q_\ast$, such that for any  $Q\in H^1(\T^2, S_0^{(2)})$ satisfying
$\|Q-Q_\ast\|_{H^1}<\beta$, we have
\begin{equation} \label{LS inequality}
\big\|L\Delta{Q}-aQ-c\tr(Q^2)Q\big\|_{(H^1)'} \geq
|\F(Q)-\F(Q_\ast)|^{1-\theta}.
\end{equation}
Here, $(H^1(\T^2, S_0^{(2)}))'$ is the dual space of $H^1(\T^2, S_0^{(2)})$.
\end{lemma}
\begin{proof}
If $Q\in S_0^{(2)}$, then it can be written into the following form
\begin{equation}
  Q(x)=\left(
      \begin{array}{cc}
        p(x) & q(x) \\
        q(x) & -p(x) \\
      \end{array}
    \right),\label{Qvector}
\end{equation}
where $p, q$ are two scalar functions defined on $\T^2$.
Now we introduce the vector $\tilde{Q}: \T^2\to  \RR^2$ defined by
$$
   \tilde{Q}=\left(
               \begin{array}{c}
                 p \\
                 q \\
               \end{array}
             \right).
$$
By direct computations, we see that
$$
 \tilde{\mathcal{F}}(\tilde{Q})=\tilde{\mathcal{F}}(p, q)\defeq
  \mathcal{F}(Q)=\int_{\T^2}\left[L(|\nabla{p}|^2+|\nabla{q}|^2)+a(p^2+q^2)+c(p^2+q^2)^2\right]dx.
$$
Then the corresponding Fr\'echet derivative of $\tilde{\F}$ with
respect to $\tilde{Q}$ in $L^2$ is given by
$$
  \frac{\delta\tilde{\F}}{\delta\tilde{Q}}=\left(
  \begin{array}{c}
  \frac{\delta\tilde{\F}}{\delta{p}} \\
  \frac{\delta\tilde{\F}}{\delta{q}} \\
  \end{array}
  \right)
=-2\left(
     \begin{array}{c}
       L\Delta{p}-ap-2c(p^2+q^2)p \\
       L\Delta{q}-aq-2c(p^2+q^2)q \\
     \end{array}
   \right)
$$
Let $\tilde{Q}_\ast=\left(
\begin{array}{c}
p_\ast \\
q_\ast \\
\end{array}
\right)$
be a critical point of $\F(\tilde{Q})$. Correspondingly, we can easily verify that
 $Q_\ast=\left(
 \begin{array}{cc}
 p_\ast & q_\ast \\
 q_\ast & -p_\ast \\
 \end{array}
 \right)$ is a critical point of $\F(Q)$. Then, applying the \L ojaciewisz-Simon inequality for vector valued functions
derived in \cite{H06}, we conclude that there exist some constants $\theta\in (0, \frac12)$ and $\beta>0$ depending on $\tilde{Q}_\ast$
(and thus $Q_\ast$), such that the following inequality holds
\begin{equation}\label{LS inequality vector form}
\Big\|\frac{\delta\tilde{\F}}{\delta\tilde{Q}}\Big\|_{(H^1(\T^2))'}
\geq
\big|\tilde{\F}(\tilde{Q})-\tilde{\F}(\tilde{Q_\ast})\big|^{1-\theta},
\end{equation}
for any $\tilde{Q}\in H^1(\T^2, \RR^2)$, provided that $\|\tilde{Q}-\tilde{Q_\ast}\|_{H^1}<\frac{\beta}{2}$.
 Therefore, our conclusion \eqref{LS inequality} is an immediate consequence of the inequality \eqref{LS inequality vector form}. The proof is complete.
\end{proof}
\begin{remark}
Lemma \ref{lemma on LS inequality} can be considered as an extended
version for matrix valued functions of Simon's result in \cite{S83}
for scalar functions. In the present case, there are two constraints
(i.e., matrix symmetry and trace free) imposed on $Q\in S_0^{(2)}$,
which might bring extra difficulties in the proof. However, due to
the special structure of the $Q$-tensor in two dimensional case
\eqref{Qvector}, the possible difficulties can be avoided by
reducing the problem to the vector case that has been treated in the
literature.
\end{remark}

The convergence of the order parameter $Q(t)$ can be proved by adapting the argument in \cite{FS} for parabolic equations,
which relies on the following analysis lemma (see e.g., \cite[Lemma 7.1]{FS})
\begin{lemma}\label{f}
Let $\theta\in(0,\frac{1}{2})$. Assume that $Z(t)\geq 0$ be a measurable function on $(0, +\infty)$, $Z(t)\in L^2(0, +\infty)$ and there exist $C>0$ and
$t_0\geq 0$ such that
\be \nonumber \int_t^{\infty} Z^2(s)ds\leq CZ(t)^{\frac{1}{1-\theta}},\quad\text{for a.e.   }t\geq t_0.\ee
Then $Z(t)\in L^1(t_0,+\infty).$
\end{lemma}

To this end, by Lemma  \ref{lemma on LS inequality}, for each element $(0,Q_\infty)\in\omega(u_0, Q_0)$, there exist some constants
$\beta_{Q_\infty}>0$ and $\theta_{Q_\infty}\in(0,\frac{1}{2})$ such that the inequality \eqref{LS inequality} holds for
\begin{equation}
Q\in \textbf{B}_{\beta_{Q_\infty}}(Q_\infty):=\Big\{Q\in H^1(\T^2,
S^{(2)}_0),\ \ \|Q-Q_\infty\|_{H^1}<\beta_{Q_\infty}\Big\}.\non
\end{equation}
The union of  balls $\{0\}\times \{\textbf{B}_{\beta_{Q_\infty}}(Q_\infty):\ (0,Q_\infty)\in\omega(u_0, Q_0)\}$ forms an open cover of $\omega(u_0, Q_0)$. Due to the compactness of $\omega(u_0, Q_0)$ in $H^1$ (see Lemma \ref{lemma on decay property}), there exists a \emph{finite }sub-cover $\{0\}\times \{\textbf{B}_{\beta_i}(Q_\infty^i):i=1,2,...,m\}$ of $\omega(u_0, Q_0)$ in $H^1$, where the constants $\beta_i, \theta_i$ corresponding to the limit point $Q_\infty^i$ (and thus a critical point of $\F(Q)$) in Lemma \ref{lemma on LS inequality} are indexed by $i$. From the definition of $\omega(u_0, Q_0)$, there exists a sufficient large $t_0>>1$ such that the global strong solution $Q(t)$ satisfies
\begin{equation}
\nonumber Q(t)\in\mathcal{U}:=\bigcup_{i=1}^m\textbf{B}_{\beta_i}(Q_\infty^i), \quad \text{for}\;\; t\geq t_0.
\end{equation}
Taking $\theta=\min_{i=1}^m\{\theta_i\}\in(0,\frac{1}{2})$, using Lemma \ref{lemma on LS inequality} and convergence of the total energy
$\mathcal{E}(t)$ (see \eqref{limE}),  we deduce, for all $t\geq t_0$,
\begin{align}
\big|\F(Q(t))-\F_\infty \big|^{1-\theta} &\leq   \big\|L\Delta{Q}(t)-aQ(t)-c\tr(Q(t)^2)Q(t)\big\|_{(H^1)'}\non\\
&\leq \|H(Q(t))\|_{L^2}.\label{ls1}
\end{align}
Therefore, we have
\begin{align}
\big( \E(t) -\F_\infty \big)^{1-\theta}&\leq \left(\frac12\|u(t)\|_{L^2}^2+\big|\F(Q(t))-\F(Q_\infty)\big|
\right)^{1-\theta} \non\\
&\leq  \Big(\frac12\|u\|_{L^2}^2+\|H(Q(t))\|_{L^2}^{\frac{1}{1-\theta}} \Big)^{1-\theta}   \non\\
&\leq C\|u(t)\|_{L^2}^{2(1-\theta)}+C\|H(Q(t))\|_{L^2} \non\\
&\leq C(\|u(t)\|_{L^2}+\|H(Q(t))\|_{L^2})\non\\
&\leq C\mathcal{A}^\frac12(t),\quad \forall\, t\geq t_0,\label{EFE}
\end{align}
in which we use the fact  $0<\theta<\frac12$ and the uniform estimate \eqref{uniES}.

On the other hand, it follows from the energy inequality \eqref{basic energy law} that
\be
 \E(t)-\F_{\infty}\geq \min\{\nu, \Gamma\}\int_{t}^{\infty}\mathcal{A}(s)ds,\quad \forall\,t\geq t_0.\label{eee0}
\ee
 As a consequence,
\bea
\int_t^{\infty}\mathcal{A}(s) ds
&\leq&  C\mathcal{A}(t)^{\frac{1}{2(1-\theta)}},\quad\forall\, t\geq t_0.\label{z}
\eea
Taking $Z(t)=\mathcal{A}(t)^\frac12$, from \eqref{z} and Lemma \ref{f} we conclude that
\be\label{z1}
\int_{t_0}^{+\infty}(\|\nabla u(t)\|_{L^2}+\| H(Q(t))\|_{L^2})dt\leq
\int_{t_0}^{+\infty} \mathcal{A}(t)^\frac12 dt <+\infty.
\ee
Then, by using the equation \eqref{Q equ} for $Q$, the uniform bounds on $\|u(t)\|_{H^1}$, $\|Q(t)\|_{H^2}$ and the Sobolev embedding theorem ($d=2$), we have
\begin{align}
\int_{t_0}^{\infty}\|Q_t(t)\|_{L^2}\,dt
&\leq \int_{t_0}^{\infty}\big(\|u\cdot\nabla{Q}\|_{L^2}+\|S(\nabla{u},
Q)\|_{L^2}+\Gamma\|H(Q)\|_{L^2}\big)\,dt \non\\
&\leq
C\int_{t_0}^{\infty}\Big[\|u\|_{L^4}\|\nabla{Q}\|_{L^4}+\|\nabla{u}\|_{L^2}\big(\|Q\|_{L^\infty}^2+1\big)
+\|H(Q)\|_{L^2}\Big]\,d\tau \non\\
&\leq C\int_{t_0}^{\infty}\big( \|\nabla{u}(t)\|_{L^2}+\|H(Q)(t)\|_{L^2}\big)\,dt
\non\\
&< +\infty,\label{inttQ}
\end{align}
which indicates that $Q(t)$ converges in $L^2(\T^2)$ for all $t\rightarrow +\infty$. Combining the sequential convergence result \eqref{sequential convergence}, it is easy to check that
\begin{equation}\label{L2 convergence}
\displaystyle\lim_{t\rightarrow+\infty}\|Q(t)-Q_\infty\|_{L^2}=0.
\end{equation}
Next, by the uniform bound on $\|Q(t)\|_{H^2}$ (see \eqref{uniES}) and \eqref{L2 convergence}, from the standard interpolation we obtain that
\begin{equation}\label{H1 convergence}
\displaystyle\lim_{t\rightarrow+\infty}\|Q(t)-Q_\infty\|_{H^1}=0.
\end{equation}
Finally, observing the following fact
\begin{align*}
\|\Delta{Q}-\Delta{Q}_\infty\|&\leq
\frac1L\|H(Q)-H(Q_\infty)\|+\frac1L\big\|aQ+c\tr(Q^2)Q-aQ_\infty-c\tr(Q_\infty^2)Q_\infty\big\|\non\\
&\leq \frac1L\|H(Q)\|+C\|Q-Q_\infty\|_{H^1},
\end{align*}
we further deduce from Lemma \ref{lemma on decay property} and \eqref{H1 convergence} that
\begin{equation}\label{H2 convergence}
\displaystyle\lim_{t\rightarrow+\infty}\|Q(t)-Q_\infty\|_{H^2}=0.
\end{equation}

\subsection{Convergence rate} In what follows, we derive uniform estimates on the convergence rate. First, the rate on
lower-order norm $\|Q(t)-Q_\infty\|_{L^2}$ follows from the \L
ojasiewicz-Simon approach (cf. \cite{HJ01}). We infer from the basic
energy law \eqref{basic energy law}, \eqref{limE} and \eqref{EFE}
that \be
 \frac{d}{dt}(\mathcal{E}(t)-\F_\infty)^\theta + C (\|\nabla{u}\|_{L^2}+\|H(Q)\|_{L^2})\leq 0, \quad \forall\, t\geq  t_0,\label{dt}
 \ee
 and
 \be
 \frac{d}{dt}(\mathcal{E}(t)-\F_\infty)+ C(\mathcal{E}(t)-\F_\infty)^{2(1-\theta)}\leq 0, \quad \forall\, t\geq  t_0.\label{ly4}
 \ee
As a consequence of \eqref{ly4}, we can deduce the rate on energy decay:
 $$ 0\leq \mathcal{E}(t)-\F_\infty \leq
 C(1+t)^{-\frac{1}{1-2\theta}},\quad \forall\, t\geq
 t_0. $$
 Then similar to \eqref{inttQ}, on $(t,+\infty)$, where $t\geq t_0$,  it follows from \eqref{dt} that
 \begin{align}
\int_{t}^{\infty}\|Q_t(s)\|_{L^2}\,ds
&\leq C\int_{t}^{\infty}\big( \|\nabla{u}(s)\|_{L^2}+\|H(Q)(s)\|_{L^2}\big)\,ds\non\\
&\leq (\mathcal{E}(t)-\F_\infty )^\theta \non\\
&\leq  C(1+t)^{-\frac{\theta}{1-2\theta}},\quad \forall\, t\geq
 t_0,
\end{align}
which further implies
 \be
    \|Q(t)-Q_\infty\|\leq C(1+t)^{-\frac{\theta}{1-2\theta}}, \quad \forall\, t\geq 0.\label{rate1}
 \ee

Higher-order estimates on the convergence rate can be achieved by constructing proper differential inequalities via
a suitable energy method (see e.g., \cite{W10} for the simplified liquid crystal system). The key idea relies on the use of the basic energy law \eqref{basic energy law} and the higher-order energy inequality \eqref{hiA}.

It follows from Lemma \ref{lemma on decay property} that the limit system of problem \eqref{navier-stokes}--\eqref{IC} takes the following form
 \begin{align}
 & \nabla P_\infty=-\lambda \nabla \cdot(\nabla Q_\infty\odot \nabla Q_\infty), \quad x\in \T^2,\label{s1a}\\
 & H(Q_\infty)=0,\quad x\in \T^2,\label{s2a}
 \end{align}
subject to periodic boundary conditions. Subtracting the
stationary problem \eqref{s1a}--\eqref{s2a} from the evolution
problem \eqref{navier-stokes}-\eqref{IC}, then testing the velocity equation by $u$ and the $Q$-equation by $-\lambda (H(Q)-H(Q_\infty))$,
respectively, adding the results together and integrating over $\T^2$, we can infer from \eqref{basic energy law} that
 \bea
  && \frac{d}{dt} \left(\frac12\|u\|_{L^2}^2+\frac{\lambda L}{2}\|\nabla Q-\nabla
  Q_\infty\|_{L^2}^2+\lambda \int_{\T^2} [f_B(Q)-f_B(Q_\infty)- f_B'(Q_\infty)(Q-Q_\infty)]dx\right)\non\\
  && \quad + \nu\|\nabla u\|_{L^2}^2+ \lambda \Gamma \|H(Q)\|_{L^2}^2\non\\
  &=& -\lambda L\int_{\T^2}u_i\nabla_j(\nabla_i (Q_\infty)_{kl}\nabla_j(Q_\infty)_{kl}) dx \non\\
  &=& -\lambda \int_{\T^2}u_i\nabla_i (Q_\infty)_{kl} (L\Delta (Q_\infty)_{kl}-[f_B'(Q_\infty)]_{kl})dx -\lambda \int_{T^2}u_i\nabla_i (Q_\infty)_{kl} [f_B'(Q_\infty)]_{kl} dx\non\\
  && -\frac{\lambda L}{2} \int_{\T^2}u_i\nabla_i|Q_\infty|^2 dx\non\\
  &=&0,
  \label{ra1}
  \eea
  where $f_B'(Q)=aQ+cQ\tr(Q^2)$.  On the other hand, testing the equation \eqref{Q equ} by $Q-Q_\infty$, from the uniform estimate \eqref{uniES},
  the H\"older inequality and the Sobolev embedding theorem ($d=2$) we conclude that
 \begin{align}
 & \frac12\frac{d}{dt}\|Q-Q_\infty\|_{L^2}^2+\Gamma L \| \nabla (Q-Q_\infty)\|_{L^2}^2\non\\
 =\ & \int_{\T^2} [-u\cdot\nabla  Q+ S(\nabla u, Q)]:( Q-Q_\infty)dx \non\\
 &- \Gamma \int_{\T^2} (f_B'(Q)-f_B'(Q_\infty)): (Q-Q_\infty) dx\non\\
 \leq \ & C\|u\|_{L^4}\|\nabla Q\|_{L^4}\|Q-Q_\infty\|_{L^2}+C\|u\|_{L^2}(\|Q\|_{L^\infty}^2+1)\|Q-Q_\infty\|_{L^2}\non\\
 & +C\int_{\T^2} \int_0^1 f_B''(sQ+(1-s)Q_\infty)(Q-Q_\infty): (Q-Q_\infty) dsdx\non\\
 \leq  \ & \epsilon_1 \|\nabla u\|_{L^2}^2+C_1\|Q-Q_\infty\|_{L^2}^2. \label{ra3}
 \end{align}
 Multiplying \eqref{ra3} by $\mu>0$ and adding the resultant to \eqref{ra1}, we get
 \begin{align}
  & \frac{d}{dt} \mathcal{Y}(t)+\left(\nu-\mu\epsilon_1\right)\|\nabla u\|_{L^2}^2  + \lambda \Gamma\|H(Q)\|_{L^2}^2 +\mu \Gamma L\|\nabla(Q-Q_\infty)\|_{L^2}^2  \non\\
 \leq\  & C_1\mu \|Q-Q_\infty\|_{L^2}^2,
 \label{ra5}
 \end{align}
 where
 \begin{align}
 \mathcal{Y}(t)&=  \frac12\|u(t)\|_{L^2}^2+\frac{\lambda L}{2}\|\nabla Q(t)-\nabla
  Q_\infty\|_{L^2}^2+  \frac{\mu}{2}\|Q(t)-Q_\infty\|_{L^2}^2\non\\
  & + \lambda \int_{\T^2} [f_B(Q(t))-f_B(Q_\infty)- f_B'(Q_\infty)(Q(t)-Q_\infty)]dx, \quad \forall\, t\geq 0. \label{y1}
  \end{align}
It follows from the Newton-Leibinz formula and \eqref{uniES} that
 \begin{align}
    & \left\vert\int_{\T^2} [f_B(Q)-f_B(Q_\infty)-f'_B(Q_\infty):(Q-Q_\infty)] dx\right\vert\non\\
 =\ & \left\vert\int_{\T^2} \int_0^1s\int_0^1f''_B(\rho(s Q+(1-s)Q_\infty)+(1-\rho)Q_\infty)(Q-Q_\infty):(Q-Q_\infty)d\rho ds dx\right\vert\non\\
 \leq\ & \|f''_B\|_{L^\infty}\|Q-Q_\infty\|_{L^2}^2 \non\\
 \leq\ & C_2\|Q-Q_\infty\|_{L^2}^2. \label{rate6}
  \end{align}
Thus we can choose $ \mu \geq  2+2\lambda C_2>0$ to see that there exists a constant $k_1>k_2>0$,
  \be k_1(\|u(t)\|_{L^2}^2+\|Q(t)-Q_\infty\|_{H^1}^2)\geq  \mathcal{Y}(t)\geq k_2(\|u(t)\|_{L^2}^2+\|Q(t)-Q_\infty\|_{H^1}^2).
 \label{y1a}
  \ee
 We now take $\epsilon_1=\frac{\nu}{2\mu}$ in \eqref{ra5}. It follows from \eqref{y1a} that there exist some
 constants $C_3,C_4>0$ such that
  \be
  \frac{d}{dt} \mathcal{Y}(t)+C_3 [\mathcal{Y}(t)+\A(t)] \leq C_4\|Q(t)-Q_\infty\|_{L^2}^2,\label{ra7}
  \ee
  which together with \eqref{rate1} yields
  \be
  \mathcal{Y}(t)\leq C(1+t)^{-\frac{2\theta}{1-2\theta}}, \quad \forall\,t\geq 0.\label{ratey}
  \ee
 In view of \eqref{y1a}, we thus obtain
  \be
   \|u(t)\|_{L^2}+\|Q(t)-Q_\infty\|_{H^1}\leq C(1+t)^{-\frac{\theta}{1-2\theta}}, \quad \forall\, t\geq 0.\label{rate2}
  \ee

At last, from the higher-order energy inequality \eqref{hiA} and the uniform estimate \eqref{uniES}, it follows that
 \be
 \frac{d}{dt}\A(t) \leq C_5\A(t).  \label{simplified higher-order inequality}
 \ee
Multiplying \eqref{simplified higher-order inequality} with
$\alpha=\frac{C_3}{2C_5}$ and adding the resultant with \eqref{ra7}, we deduce
 \be \frac{d}{dt}[\mathcal{Y}(t)+\alpha\A(t)]+C_6[\mathcal{Y}(t)+\alpha \A(t) ] \leq
C(1+t)^{-\frac{2\theta}{1-2\theta}}. \label{ddecaA}
 \ee
 As a consequence,
 \be
 \mathcal{Y}(t)+\alpha \A(t) \leq C(1+t)^{-\frac{2\theta}{1-2\theta}}, \ \ \ \ \ \forall\, t \geq 0,\non
 \ee
which together with the fact  $\mathcal{Y}(t) \geq 0$ (see \eqref{y1a}) yields
   \be
   \A(t)\leq C(1+t)^{-\frac{2\theta}{1-2\theta}}, \quad \forall\, t\geq 0.\label{rate3a}
  \ee
Then from the definition of $\A(t)$ and estimates \eqref{rate2}, \eqref{rate3a}, we can see that
 \be
 \|\nabla u(t)\|_{L^2}+ \|\Delta Q(t)-\Delta Q_\infty\|_{L^2} \leq C(1+t)^{-\frac{\theta}{1-2\theta}}, \quad \forall\, t\geq 0.\label{rate4a}
  \ee
Collecting the estimates \eqref{rate2} and \eqref{rate4a}, we arrive at the conclusion \eqref{covrate}.

The proof of Theorem \ref{long-time} is complete.


\section{Appendix}
\setcounter{equation}{0}
The following calculations hold for both two and three dimensional cases.
\begin{lemma}
Suppose $d=2, 3$. Let $(u, Q)$ be a smooth solution to the problem \eqref{navier-stokes}-\eqref{IC}. Define the quantity
\begin{equation}
\A(t)=\|\nabla u\|_{L^2}^2+\lambda \|H(Q)\|_{L^2}^2.\label{A1}
\end{equation}
Then we have the following equality
\begin{eqnarray}
&&\frac12\frac{d}{dt}\A(t)+ \nu\|\Delta{u}\|_{L^2}^2+\lambda \Gamma\|\nabla H\|_{L^2}^2\non\\
& =&\int_{\T^d}\big(u\cdot\nabla{u},\Delta{u}\big)\,dx
-2\lambda\int_{\T^d} \nabla_l u_k\nabla_l\nabla_kQ_{ij} H_{ij} dx
+\frac{\lambda}{L}\int_{\T^d} u_k\nabla_k F_{ij} H_{ij} dx\non\\
&& - 2\lambda\int_{\T^d} \nabla_j u_i(\nabla_lQ_{kj}\nabla_lH_{ik}-\nabla_l Q_{ik}\nabla_lH_{kj}) dx
-\lambda\int_{\T^d} \nabla_j u_i(\Delta Q_{kj}H_{ik}-\Delta Q_{ik}H_{kj})dx\non\\
&& +\lambda\xi\int_{\T^d}\big(D\Delta{Q}+\Delta{Q}D):H dx
+ 4\lambda\xi \int_{\T^d} \nabla_lD_{ik}\nabla_lQ_{kj}H_{ij}\,dx\non\\
&& -2\lambda\xi\int_{\T^d} \Delta\big(Q_{kl}Q_{ji}\big)\nabla_ju_iH_{kl}\,dx
-4\lambda\xi\int_{\T^d}\nabla_m\big(Q_{kl}Q_{ji}\big)\nabla_m\nabla_ju_iH_{kl}\,dx\non\\
&& -\lambda\int_{\T^d}\frac{\partial{F(Q)}}{\partial Q} (u\cdot \nabla Q) : Hdx
+\lambda\int_{\T^d}\frac{\partial{F(Q)}}{\partial Q}S(\nabla u, Q): H dx\non\\
&& +\lambda\Gamma \int_{\T^d}\frac{\partial{F(Q)}}{\partial Q} H : H dx\non\\
&:=& \sum_{i=1}^{12} J_i,\label{DAT}
\end{eqnarray}
where
\begin{equation}
F(Q)=-aQ+b\left(Q^2-\frac1d\mathrm{tr}(Q^2)\mathbb{I}\right)-cQ\tr(Q^2)=-\frac{\partial f_B(Q)}{\partial Q}-\frac{b}{d}\mathrm{tr}(Q^2)\mathbb{I}.\label{def of F}
\end{equation}
\end{lemma}
\begin{proof}
Using the equations \eqref{navier-stokes}, \eqref{Q equ} and integration by parts, we have
\begin{eqnarray}
&& \frac12\frac{d}{dt}\|\nabla u\|_{L^2}^2+ \nu\|\Delta{u}\|_{L^2}^2\non\\
&=& \int_{\T^d} (u\cdot\nabla{u},\Delta{u})\,dx-\lambda\int_{\T^d}(\nabla\cdot\sigma,
\Delta{u})\,dx-\lambda\int_{\T^d}(\nabla\cdot\tau, \Delta{u})\,dx\non\\
&:=& I_1+I_2+I_3,
\end{eqnarray}
where
\begin{equation}
I_2= -\lambda\int_{\T^d}\nabla_j(Q_{ik}H_{kj}-H_{ik}Q_{kj})\Delta{u}_i\,dx,\label{A-2}
\end{equation}
and
\begin{align}\label{A-3}
I_3&=\lambda L\int_{\T^d} \nabla_j(\nabla_iQ_{kl}\nabla_j{Q}_{kl})\Delta{u}_i\,dx
+\lambda\xi\int_{\T^d}\nabla_j\Big(Q_{ik}H_{kj}+H_{ik}Q_{kj}+\frac{2}{d}H_{ij}\Big)\Delta{u}_i\,dx\non\\
&\quad-2\lambda\xi\int_{\T^d} \nabla_j\big(Q_{kl}H_{kl}Q_{ij}+\frac{1}{d} Q_{kl}H_{kl}\delta_{ij}\big)\Delta{u}_i\,dx\non\\
&:=I_{3a}+I_{3b}+I_{3c}.
\end{align}
As in \cite[A.3]{ADL15}, using the incompressibility condition $\nabla\cdot{u}=0$ , the definition of $F(Q)$ (see \eqref{def of F}) and the fact $\nabla \mathrm{tr}(Q)=0$, we get
\begin{eqnarray}
I_{3a}&=& \lambda \int_{\T^d} \nabla_iQ_{kl}(H_{kl}-F_{kl})\Delta{u}_i\,dx+ \lambda L\int_{\T^d} \nabla_j\nabla_iQ_{kl}\nabla_j{Q}_{kl}\Delta{u}_i\,dx\non\\
&=&  \lambda \int_{\T^d} \nabla_iQ_{kl}H_{kl}\Delta{u}_i\,dx+\lambda \int_{\T^d} \nabla f_B(Q)\cdot  \Delta{u}\,dx + \frac{\lambda b}{d}\int_{\T^d} \mathrm{tr}(Q^2)\nabla \mathrm{tr}(Q) \cdot \Delta u dx   \non\\
&& + \frac{\lambda L}{2}\int_{\T^d} (\nabla |\nabla Q|^2) \cdot \Delta{u}\,dx\non\\
&=& \lambda \int_{\T^d} \nabla_iQ_{kl}H_{kl}\Delta{u}_i\,dx.\label{A-3a}
\end{eqnarray}
Using the symmetry of $Q$ and $H(Q)$, and the basic algebra for arbitrary matrices $A, B, C \in \RR^{d\times d}$
\begin{equation}
(AB):C=B:(A^TC)=A:(CB^T),\label{ABC}
\end{equation}
 we have
\begin{eqnarray}
I_{3b}&=& -\frac{2\lambda \xi}{d}\int_{\T^d}H_{ij}\nabla_j\Delta u_i\,dx
-\lambda \xi\int_{\T^d}(Q_{ik}H_{kj}+H_{ik}Q_{kj})\nabla_j\Delta u_i\,dx\non\\
&=&   -\frac{2\lambda \xi}{d}\int_{\T^d}H_{ij}\Delta{D}_{ij}
\,dx
-\lambda \xi\int_{\T^d}Q_{ik}H_{kj}\Delta{D}_{ij}\,dx\non\\
&&-\lambda \xi\int_{\T^d}H_{ik}Q_{kj}\Delta{D}_{ij}\,dx \non\\
&=&-\lambda \xi\int_{\T^d} \big(\Delta{D}Q+Q\Delta{D} +\frac{2}{d} \Delta{D}\big): H\,dx.\label{I-3b}
\end{eqnarray}
By the incompressibility condition $\nabla\cdot{u}=0$, it holds
\begin{eqnarray}
I_{3c}&=& -2\lambda \xi\int_{\T^d} \nabla_j\big(Q_{kl}H_{kl}Q_{ij})\Delta u_idx-\frac{2\lambda \xi}{d}\int_{\T^d}\nabla_i(Q_{kl} H_{kl})\Delta{u}_i\,dx
\non\\
&=& -2\lambda \xi\int_{\T^d} \nabla_j\big(Q_{kl}H_{kl}Q_{ij})\Delta u_idx.\label{I-3c}
\end{eqnarray}
On the other hand, we have
\begin{align}\label{A-1}
&\quad \frac{\lambda}{2}\frac{d}{dt}\|H(Q)\|_{L^2}^2+\lambda \Gamma\|\nabla H(Q)\|_{L^2}^2\non\\
&=\lambda L\int_{\T^d}(\Delta{Q}_t : H(Q))\,dx+ \lambda \int_{\T^d} \partial_t F(Q) : H(Q) \,dx+ \lambda \Gamma\|\nabla H(Q)\|_{L^2}^2\non\\
&=\lambda L\int_{\T^d}({Q}_t : \Delta H(Q))\,dx+ \lambda \int_{\T^d} \partial_t F(Q) : H(Q) \,dx+ \lambda \Gamma\|\nabla H(Q)\|_{L^2}^2\non\\
&= -\lambda \int_{\T^d} u_k\nabla_k Q_{ij}\Delta H_{ij}dx
+ \lambda \int_{\T^d} S(\nabla u, Q):\Delta H dx
+ \lambda \int_{\T^d} \partial_t F(Q) : H(Q) \,dx\non\\
&:= I_4+I_5+I_6.
\end{align}
Then by the incompressibility condition $\nabla\cdot{u}=0$ and \eqref{Q equ} we have
\begin{eqnarray}
I_4
&=& -\lambda \int_{\T^d} u_k\nabla_k Q_{ij}\Delta H_{ij}dx\non\\
&=& \lambda \int_{\T^d} (\nabla_l u_k\nabla_k Q_{ij}+u_k\nabla_l\nabla_k Q_{ij}) \nabla_l H_{ij}dx\non\\
&=& -\lambda \int_{T^d} \Delta u_k \nabla_k Q_{ij} H_{ij} dx
-2\lambda \int_{\T^d} \nabla_l u_k\nabla_l\nabla_kQ_{ij} H_{ij} dx\non\\
&& -\frac{\lambda }{L}\int_{\T^d} u_k\nabla_k(H_{ij}-F_{ij})H_{ij} dx\non\\
&=& -\lambda \int_{\T^d} \Delta u_k \nabla_k Q_{ij} H_{ij} dx
-2\lambda \int_{\T^d} \nabla_l u_k\nabla_l\nabla_kQ_{ij} H_{ij} dx\non\\
&& +\frac{\lambda }{L}\int_{\T^d} u_k\nabla_k F_{ij} H_{ij} dx\non\\
&:=& I_{4a}+I_{4b}+I_{4c}.\label{A-4-2}
\end{eqnarray}
\begin{eqnarray}
I_5 &=& \lambda \int_{\T^d}(\xi D+\Omega)\big(Q+\frac1d \Id
\big)+\big(Q+\frac1d \Id \big)(\xi D-\Omega)-2\xi\big(Q+\frac1d \Id
\big)\tr(Q\nabla{u}):\Delta H dx\non\\
&=& \lambda \int_{\T^d} (\Omega Q-Q\Omega):\Delta Hdx
+ \lambda \xi\int_{\T^d} \big(DQ+QD+\frac{2}{d} D\big):\Delta H dx\non\\
&& -2\lambda \xi\int_{\T^d}\mathrm{tr}(Q\nabla u)\big(Q+\frac{1}{d}\mathbb{I}\big):\Delta H dx\non\\
&:=& I_{5a}+I_{5b}+I_{5c}.\non
\end{eqnarray}
Using the symmetry properties of $Q$, $H$ and \eqref{ABC}, after integration by parts, it is easy to check that
\begin{eqnarray}
I_{5a}&=&\lambda \int_{\T^d} (\Omega{Q}-Q\Omega): \Delta H  dx\non\\
&=& \frac{\lambda}{2} \int_{\T^d}(\nabla u Q-\nabla^T u Q-Q\nabla u+Q\nabla^T u):\Delta H dx\non\\
&=& \lambda \int_{\T^d} \nabla u: (\Delta H Q-Q\Delta H) dx\non\\
&=& \lambda \int_{\T^d}\Delta(\nabla_j u_iQ_{kj})H_{ik}\,dx
-\lambda \int_{\T^d}\Delta(\nabla_j u_iQ_{ik})H_{kj}\,dx\non\\
&=& \lambda \int_{\T^d} \Delta \nabla_j u_i(Q_{kj}H_{ik}-Q_{ik}H_{kj}) dx
+ 2\lambda \int_{\T^d} \nabla_l\nabla_j u_i(\nabla_lQ_{kj}H_{ik}-\nabla_l Q_{ik}H_{kj}) dx\non\\
&& +\lambda \int_{\T^d} \nabla_j u_i(\Delta Q_{kj}H_{ik}-\Delta Q_{ik}H_{kj})dx\non\\
&=& \lambda \int_{\T^d} \nabla_j(Q_{ik}H_{kj}-Q_{kj}H_{ik}) \Delta  u_i  dx
- 2\lambda \int_{\T^d} \nabla_j u_i(\nabla_lQ_{kj}\nabla_lH_{ik}-\nabla_l Q_{ik}\nabla_lH_{kj}) dx\non\\
&& -\lambda \int_{\T^d} \nabla_j u_i(\Delta Q_{kj}H_{ik}-\Delta Q_{ik}H_{kj})dx.\label{I-5a}
\end{eqnarray}
Moreover, using integration by parts, we have
\begin{align}\label{A-4-4}
I_{5b} &= \lambda \xi\int_{\T^d} \Delta \big(DQ+QD+\frac{2}{d} D\big):H dx\non\\
&=\lambda \xi\int_{\T^d} \big(\Delta{D}Q+Q\Delta{D} +\frac{2}{d} \Delta{D}\big): H\,dx\non\\
&\quad +\lambda \xi\int_{\T^d}\big(D\Delta{Q}+\Delta{Q}D):H dx+ 4\lambda \xi \int_{\T^d} \nabla_lD_{ik}\nabla_lQ_{kj}H_{ij}\,dx.
\end{align}
Next, for $I_{5c}$, using the property $\mathrm{tr}(H(Q))=0$ and after integration by parts, we obtain that
\begin{align}\label{A-4-5}
I_{5c}&=-2\lambda \xi\int_{\T^d} \Delta\big[\tr(Q\nabla{u})Q\big]: H\,dx\non\\
&=-2\lambda \xi\int_{\T^d} \Delta\big(Q_{kl}Q_{ji}\nabla_ju_{i} \big)H_{kl}\,dx\non\\
&=-2\lambda \xi\int_{\T^d}{Q}_{kl}Q_{ji}\nabla_j\Delta{u}_{i}H_{kl}\,dx
   -2\lambda \xi\int_{\T^d} \Delta\big(Q_{kl}Q_{ji}\big)\nabla_ju_iH_{kl}\,dx\non\\
&\quad-4\lambda \xi\int_{\T^d}\nabla_m\big(Q_{kl}Q_{ji}\big)\nabla_m\nabla_ju_iH_{kl}\,dx\non\\
&=2\lambda \xi\int_{\T^d}\nabla_j\big(Q_{kl}H_{kl}Q_{ji}\big)\Delta{u}_{i}\,dx
   -2\lambda \xi\int_{\T^d} \Delta\big(Q_{kl}Q_{ji}\big)\nabla_ju_iH_{kl}\,dx\non\\
&\quad-4\lambda \xi\int_{\T^d}\nabla_m\big(Q_{kl}Q_{ji}\big)\nabla_m\nabla_ju_iH_{kl}\,dx.
\end{align}
Finally, using the equation \eqref{Q equ}, the term $I_6$ can be expressed as follows
\begin{eqnarray}
I_6&=& \lambda \int_{\T^d} \frac{\partial{F(Q)}}{{\partial Q}}\partial_t Q : H(Q) \,dx\non\\
   &=& \lambda \int_{\T^d} \frac{\partial{F(Q)}}{\partial Q} (-u\cdot \nabla Q+ S(\nabla u, Q)+\Gamma H(Q)) : H(Q) dx.\non
\end{eqnarray}
In summary, we find the following special cancellations between those highly nonlinear terms:
\begin{itemize}
\item[(a)] the term $I_2$ (i.e., \eqref{A-2}) cancels with the first term in \eqref{I-5a},
\item[(b)] the term $I_{3a}$ (i.e., \eqref{A-3a}) cancels with the term $I_{4a}$ in \eqref{A-4-2},
\item[(c)] the term $I_{3b}$ (i.e., \eqref{I-3b}) cancels with the first term in \eqref{A-4-4},
\item[(d)] the term $I_{3c}$ (i.e., \eqref{I-3c}) cancels with the first term in \eqref{A-4-5}.
\end{itemize}
Taking into account the above cancellation relations, we can easily conclude \eqref{DAT}.
\end{proof}
\begin{remark}
We note that in the above cancellations, the relation (a) is the same as for the simpler case with $\xi=0$ (see e.g., \cite{ADL15}). However, for the general case $\xi\neq 0$, we have found extra relations (b)-(d) between higher-order nonlinear terms of the full Navier-Stokes and Q-tensor system \eqref{navier-stokes}-\eqref{IC}.
\end{remark}

\bigskip

\noindent \textbf{Acknowledgements.}
C. Cavaterra and E. Rocca are supported by the FP7-IDEAS-ERC-StG \#256872 (EntroPhase)
and by GNAMPA (Gruppo Nazionale per l'Analisi Matematica, la Probabilit\`a
e le loro Applicazioni) of INdAM (Istituto Nazionale di Alta
Matematica).
H. Wu is partially supported by NNSFC
under the grant No. 11371098 and Shanghai Center for Mathematical
Sciences of Fudan University. X. Xu is supported by the
start-up fund from the Department of Mathematics and Statistics at
Old Dominion University.


\end{document}